\documentclass[12pt]{article}

\usepackage{amssymb}
\usepackage{graphics} 
\usepackage{color}
\usepackage{url}
\usepackage{hyperref}
\usepackage[numbers]{natbib}

\begin{document}

\definecolor{dred}{rgb}{0.85,0.05,0.15}

\newenvironment {proof}{{\noindent\bf Proof.}}{\hfill $\Box$ \medskip}

\newtheorem{theorem}{Theorem}[section]
\newtheorem{lemma}[theorem]{Lemma}
\newtheorem{condition}[theorem]{Condition}
\newtheorem{proposition}[theorem]{Proposition}
\newtheorem{remark}[theorem]{Remark}
\newtheorem{hypothesis}[theorem]{Hypothesis}
\newtheorem{corollary}[theorem]{Corollary}
\newtheorem{example}[theorem]{Example}
\newtheorem{definition}[theorem]{Definition}

\renewcommand {\theequation}{\arabic{section}.\arabic{equation}}

\newcommand{\red}[1]{\textcolor{red}{#1}}

\def \non{{\nonumber}}
\def \hat{\widehat}
\def \tilde{\widetilde}
\def \bar{\overline}
\def \P{\mathbb{P}}
\def \R{\mathbb{R}}
\def \1{\bf  1}

\title{\large {\bf  Markov selection for constrained martingale problems}}
                                                       
\author{\small \begin{tabular}{ll}                        
Cristina Costantini & Thomas G. Kurtz   \\\ 
Dipartimento di Economia & Departments of Mathematics and Statistics  \\\      
Universit\`a di Chieti-Pescara & University of Wisconsin - Madison \\\                 
v.le Pindaro 42 & 480 Lincoln Drive  \\\                                                          
65127 Pescara, Italy& Madison, WI  53706-1388, USA   \\\                         
c.costantini@unich.it & kurtz@math.wisc.edu    \\
& \url{http://www.math.wisc.edu/~kurtz/}  \\\                      
\end{tabular}}

\date{November 10, 2019}

\maketitle

\begin{abstract}
Constrained Markov processes, such as reflecting diffusions, 
behave as an unconstrained process in the  interior of a 
domain but upon reaching the boundary are controlled in 
some way so that they do not leave the closure of the 
domain.  In this paper, the behavior in the interior is 
specified by a generator of a Markov process, and the 
constraints are specified by a controlled generator.  
Together, the generators define a 
{\em constrained martingale problem}.  The 
desired constrained processes are constructed by first 
solving a simpler {\em controlled martingale problem\/} and then
obtaining the desired process as a time-change of the 
controlled process. 

As for ordinary martingale problems, 
it is rarely obvious that the process constructed in this 
manner is unique.  The primary goal of the paper is to 
show that from among the processes constructed in this 
way one can ``select'', in the sense of Krylov, a strong Markov 
process.  Corollaries to these constructions include the 
observation that uniqueness among strong Markov solutions implies 
uniqueness among all solutions.

These results provide useful tools for proving 
uniqueness for constrained processes including 
reflecting diffusions.

The constructions also yield
viscosity semisolutions of the resolvent equation and, 
if uniqueness holds, a viscosity solution,  
without proving a comparison principle.

We illustrate our results by applying them to reflecting diffusions  
in piecewise smooth domains. We prove existence of a 
strong Markov solution to the SDE with reflection, 
under conditions more general than in \cite{DI93}: 
In fact our conditions are known to be optimal in the case of simple, convex 
polyhedrons with constant direction of reflection on each 
face (\cite{DW95}). 
We also indicate how the results can be applied to 
processes with Wentzell boundary conditions and nonlocal 
boundary conditions.



\end{abstract}

\setcounter{equation}{0}
\section{Introduction} \label{markselect}
Let $A$ be an operator determining a Markov process $X$ 
with state space $E$ as the solution of 
the martingale problem in which
\begin{equation}M_f(t)=f(X(t))-f(X(0))-\int_0^tAf(X(s))ds\label{mgp0}\end{equation}
is required to be a martingale with respect to a 
filtration $\{{\cal F}_t\}$ for all $f\in {\cal D}(A)$, the domain of $
A$.  The 
study of stochastic processes that behave like the 
process determined by $A$ when in an open subset $E_0\subset E$, 
are constrained to stay in $\bar {E}_0$, and must behave in a 
prescribed way on $\partial E_0$, is classically carried out by 
restricting the domain ${\cal D}(A)$ by specifying boundary 
conditions, typically of the form $Bf(x)=0$ for $x\in\partial E_0$ 
for some operator $B$.  Then $X$ is required to remain in 
$\bar {E}_0$ and (\ref{mgp0}) is required to be a martingale for 
all functions in $\{f\in {\cal D}(A):Bf(x)=0,x\in\partial E_0\}$.  This 
approach to constrained Markov processes, however, 
frequently introduces difficult analytical problems in 
identifying a set of functions both satisfying the boundary 
conditions and large enough to characterize the process.  

An alternative approach by Stroock and Varadhan 
\cite{SV71} introduces a submartingale problem which 
weakens the restriction on the domain of $A$ to the 
requirement that
 $Bf(x)\geq 0$ for $x\in\partial E_0$ and then requires that for all such $
f\in {\cal D}(A)$, 
(\ref{mgp0}) is a submartingale.
This approach has been used to 
great effect by a number of authors. See, for example, 
\cite{Wei81,KR14,KR17}.

Restrictions on the values of $Bf$ on the boundary are 
dropped altogether in \cite{Kur90, Kur91} at the cost of 
introducing a boundary process $\lambda$ that, in the simplest 
settings, measures the amount of time the process 
spends on the boundary in the sense that $\lambda$ is 
nondecreasing and increases only when $X$ (or more 
precisely $X(\cdot -)$) is on the boundary.  Then $X$ is required 
to take values in $\bar {E}_0$ and for each $f\in {\cal D}(A)\cap 
{\cal D}(B)$,
\begin{equation}M_f(t)=f(X(t))-f(X(0))-\int_0^tAf(X(s))ds-\int Bf
(X(s-))d\lambda (s)\label{mgp}\end{equation}
is required to be a martingale.  As we will see, the form of the 
boundary term may be more complicated than this. 
A process that satisfies the{\em s\/}e requirements is a solution 
of the {\em constrained martingale problem.  }
Clearly, every solution of the constrained 
martingale problem is also a solution of the 
submartingale problem. This approach, or 
the corresponding one for stochastic equations, has 
been used, for example, in \cite{DW95, CK06, CK18}. 

Whether the submartingale problem approach or the 
constrained martingale problem approach is used, the 
critical issue is uniqueness of the solution, which is 
still an open question for many examples (see e.g. 
\cite{KKLW09,KW12}).

The primary goal of this paper is to prove a Markov 
selection theorem for solutions of constrained martingale 
problems.  Beyond the intrinsic interest, this selection 
theorem is frequently a crucial ingredient in proving 
uniqueness for constrained martingale problems and 
hence uniqueness for semimartingale reflecting Brownian 
motion (see, for example, \cite{KW91,TW93,DW95}) and 
reflecting diffusions.  

In the 
unconstrained case, the Markov selection theorem 
ensures the existence of strong Markov solutions to the 
martingale problem.   The construction of the strong Markov 
solution also ensures that 
uniqueness among strong Markov 
solutions implies uniqueness among all solutions. See   
\cite{SV79}, Theorems 12.2.3 and 12.2.4, for diffusions and 
\cite{EK86}, Theorem 4.5.19, for general martingale 
problems. All these results follow \cite{Kry73}.
The observation that uniqueness among strong Markov solutions 
implies uniqueness among all solutions provides 
a key tool in uniqueness arguments.  Unfortunately, 
these results do not apply immediately to solutions of 
submartingale or constrained martingale problems.  

We construct solutions of the constrained martingale 
problem by time-changing solutions of a {\em controlled }
{\em martingale problem\/} (Sections \ref{ClMP} and \ref{CMP}). 
Solutions of the controlled martingale problem evolve on a 
slower time scale and may take values in all of $E$. Their 
behavior in $E_0^c$ is determined by the operator $B$. 
Since solutions of the controlled martingale problem capture the 
intuition behind the controls that constrain the solution, 
we will refer to solutions of the constrained martingale 
problem that arise as time-changes of solutions of the 
controlled martingale problem as {\em natural}. 
We cannot rule out the possibility that there are solutions 
of the constrained martingale problem which are not 
natural, but, under very general conditions, uniqueness 
for natural solutions implies uniqueness for all solutions. 
See Remark \ref{natunq}. 

In Section \ref{clmp}, we introduce the controlled 
martingale problem and discuss properties of the 
collection of solutions.  In particular, we 
prove weak compactness of the 
collection of solutions. In Section \ref{CMP}, we introduce 
the time-changed process. Under mild conditions, the time-changed 
process is a natural solution of the constrained martingale 
problem. We note however that, even when it is not, 
the time-changed process still models a 
process constrained in $\bar {E}_0$, with behavior in the interior 
determined by $A$ and constraints determined by $B$. 

In Section \ref{selection} we prove that there exists 
a natural strong Markov solution of the constrained martingale 
problem (Theorem \ref{uniqueness} and Corollary 
\ref{M-const}) and that uniqueness among natural strong Markov  
solutions implies uniqueness among all natural solutions 
(Corollary \ref{Muniq}). 

In Section \ref{sectvisc}, we discuss connections between 
solutions of the constrained martingale problem and 
viscosity semisolutions of the corresponding resolvent 
equation. In particular, generalizing the results of 
Section 5 of \cite{CK15}, we see that existence of a 
comparison principle for the viscosity semisolutions
implies uniqueness for natural solutions of 
the constrained martingale 
problem. Conversely, uniqueness of natural solutions of the
constrained martingale problem gives a viscosity solution of the 
resolvent equation. Thus one can obtain existence of a 
viscosity solution from purely probabilistic arguments, 
without first proving 
a comparison principle for the resolvent equation. 

In Section \ref{reflect} we apply the results of Section 
\ref{selection} to diffusion processes in piecewise 
smooth domains of $\R^d$ with varying, oblique directions 
of reflection on each face.  Existence and uniqueness 
results for these processes have been obtained by many 
authors (\cite{TW93, DW95} 
for convex polyhedrons with constant direction of 
reflection on each face, \cite{Tan79, LS84, Cos92, DI93} for 
nonpolyhedral domains, etc.). For nonpolyhedral domains, 
\cite{DI93} is perhaps the most general result, 
but it still requires a condition that is not satisfied in 
some very natural examples (see Example \ref{ex}) or is 
difficult to verify in other ones (see e.g. \cite{KW12}).  
In addition, \cite{DI93} does not cover the case of cusp 
like singularities, such as in \cite{KKLW09} 
(in dimension 2, cusp like singularities are covered by \cite{CK18}).  
In \cite{TW93} and \cite{DW95} 
a key point in proving uniqueness is the fact that 
there exist strong Markov processes that satisfy the definition 
of reflecting diffusion 
and that uniqueness among these strong 
Markov processes implies uniqueness.  
By the results of Section \ref{selection}, we obtain 
existence of a strong Markov natural solution of the 
constrained martingale problem under 
conditions that coincide with 
those of \cite{DW95} in the case of simple, 
convex polyhedrons with constant direction of 
reflection on each face (see Remark \ref{polyhedron}). 
In this case, \cite{DW95} have shown that these 
conditions are necessary for existence of a semimartingale 
reflecting Brownian motion. 
Under the same assumptions, 
the results of Section \ref{selection} ensure also that 
uniqueness among strong Markov natural solutions implies 
uniqueness among all natural solutions. 
Moreover we show that the set of natural solutions 
of the constrained martingale problem coincides with the 
set of weak solutions to the corresponding stochastic 
differential equation with reflection 
(Theorem \ref{equiv}). 

Further examples of application of the results of Section 
\ref{selection} are presented in Section \ref{applications}.

\subsection{Notation}\label{notation}

\begin{itemize}
\item[]For a metric space $(E,r)$, ${\cal B}(E)$ will denote the $
\sigma$-algebra of 
Borel subsets of $E$, $B(E)$ will denote the set of bounded, 
Borel measurable functions on $E$, and $\|\cdot\|$ will denote the 
supremum norm on $B(E)$.

\item[]${\cal P}(E)$ will denote the set of probability measures on 
$(E,{\cal B}(E))$. For $F\in {\cal B}(E)$, with a slight abuse of notation, 
${\cal P}(F)$ will denote $\{P\in {\cal P}(E):\,P(F)=1\}$.

\item[]For $x\in E$ and $F\in {\cal B}(E)$, $d(x,F)$ will denote the 
distance from $x$ to $F$, that is, $d(x,F)=\inf_{y\in F}r(x,y)$.

\item[]${\bf 1}$ will denote the function identically equal 
to $1$ and, for $F\in {\cal B}(E)$, ${\bf 1}_F$ will denote the indicator function 
of $F$.

\item[]$|I|$ will denote the cardinality of a finite set $I$.

\item[]For any function or operator, ${\cal R}(\cdot )$ will denote the range 
and ${\cal D}(\cdot )$ the domain.

\item[] ${\cal L}(\cdot )$ will denote the distribution of a stochastic process 
or a random variable.

\item[]If $Z$ is a stochastic process defined on an arbitrary 
probability space, $\{{\cal F}^{Z}_t\}$ will denote the filtration generated 
by $Z$.

\item[]If $Z$ is a stochastic process defined on an arbitrary 
filtered probability space, $Z$ will also denote the 
canonical process defined on the path space. 
$\{{\cal B}_t\}$ will denote the filtration generated by the canonical 
process. 

\end{itemize}

\setcounter{equation}{0}
\section{Controlled martingale problems}\label{ClMP}

We use the control formulation of constrained martingale 
problems given in \cite{Kur91} rather than the earlier 
version given in \cite{Kur90} that was based on 
``patchwork'' martingale problems.  The control 
formulation may be less intuitive, but it is more general 
and notationally 
simpler, and models described in the earlier manner can 
be translated to the control formulation.

Let $E$ be a compact metric space, and let $E_0$ be an open 
subset of $E$.  The requirement that $E$ be compact is not 
particularly restrictive since, for example, for most 
processes in $\R^d$, one can take $E$ to be the 
one-point compactification of $\R^d$.  Let $A\subset C(E)\times C(E)$ with 
$(1,0)\in A$.  

Let $U$ also be a compact metric space, and let 
$\Xi$ be a closed subset of $E_0^c\times U$.  For each $x\in E_0^
c$, 
let $\xi_x\equiv \{u:(x,u)\in\Xi \}$ be the set of controls that are 
admissible at $x$, and define  $F_1\equiv \{x\in E_0^c:\,\xi_x\neq
\emptyset \}$ which is the set of 
points at which a control exists.  Let $B\subset C(E)\times C(\Xi 
)$ 
with $(1,0)\in B$.  Using $A$ and $B$, we define a controlled process 
$Y$ that outside $E_0$ evolves on a slower time scale than 
the desired process $X$.  Like $X$, inside $E_0$ the behavior of 
$Y$ is determined by $A$, and outside $E_0$ the behavior of $Y$ 
is determined by $B$. In particular, $Y$ 
may take values in $\bar {E}_0\cup F_1$.

Let ${\cal L}_U$ be the space of measures on $[0,\infty )\times U$ 
such that $\mu ([0,t]\times U)<\infty$ for all $t>0$.  ${\cal L}_
U$ is
topologized so that  
$\mu_n\in {\cal L}_U\rightarrow\mu\in {\cal L}_U$ if and only if 
\[\int_{[0,\infty )\times U}f(s,u)\mu_n(ds\times du)\rightarrow\int_{
[0,\infty )\times U}f(s,u)\mu (ds\times du)\]
for all 
continuous $f$ with compact support in $[0,\infty )\times U$. It is 
possible to define a metric on ${\cal L}_U$ that induces the above 
topology and makes ${\cal L}_U$ into a complete, separable metric 
space. We will say that an ${\cal L}_U$-valued random variable 
$\Lambda_1$ is adapted to a filtration $\{{\cal F}_t\}$ if 
\[\Lambda_1([0,\cdot ]\times C)\mbox{\rm \ is }\{{\cal F}_t\}-\mbox{\rm adapted}
,\quad\forall C\in {\cal B}(U).\]

\begin{definition}\label{clmp}
$(Y,\lambda_0,\Lambda_1)$ is a solution of the 
{\em controlled martingale problem }
for $(A,E_0,B,\Xi )$, if $Y$ is a process in $D_E[0,\infty )$, 
$\lambda_0$ is nonnegative and nondecreasing and increases only 
when $Y\in\bar {E}_0$, $\Lambda_1$ is a random measure in ${\cal L}_
U$ such that 
\begin{equation}\lambda_1(t)\equiv\Lambda_1([0,t]\times U)=\int_{[0,t]\times U}
{\bf 1}_{\Xi}(Y(s),u)\Lambda_1(ds\times du),\label{Lam1supp}\end{equation}
\[\lambda_0(t)+\lambda_1(t)=t,\]
and there exists a filtration $\{{\cal F}_t\}$ such that $Y$, $\lambda_
0$, and $\Lambda_1$ are 
$\{{\cal F}_t\}$-adapted and 
\begin{equation}f(Y(t))-f(Y(0))-\int_0^tAf(Y(s))d\lambda_0(s)-\int_{
[0,t]\times U}Bf(Y(s),u)\Lambda_1(ds\times du)\label{mgp1}\end{equation}
is an $\{{\cal F}_t\}$-martingale for all $f\in {\cal D}\equiv {\cal D}
(A)\cap {\cal D}(B)$. By the 
continuity of $f$, we can 
assume, without loss of generality, that $\{{\cal F}_t\}$ is right 
continuous. 
\end{definition}

\begin{remark}\label{exist}
To get some intuition on $\lambda_0$ and $\Lambda_1$, consider the case 
in which $A$ is a bounded Markov process generator and at each point 
$x\in (E_0)^c$ there is exactly one control $u(x)$, so $B$ is the
bounded Markov process generator that, at $x$, produces a jump 
$u(x)$. Then $Y$ is the pure jump process with generator 
$Af(x)\,{\bf 1}_{E_0}(x)+Bf(x,u(x))\,{\bf 1}_{(E_0)^c}(x)$. 
$\lambda_0(t)$ and $\Lambda_1([0,t]\times C)$ are the time that $Y$ spends in $
E_0$ and the 
time that $Y$ spends in $(E_0)^c$ while the control lies in $C$, 
respectively, i.e. 
\[\lambda_0(t):=\int_0^t{\bf 1}_{E_0}(Y(s))\,ds,\quad\Lambda_1([0
,t]\times C):=\int_0^t{\bf 1}_{(E_0)^c}(Y(s))\,{\bf 1}_C(u(Y(s)))\,
ds.\]
For general $A$ and $B$, frequently $(Y,\lambda_0,\Lambda_1)$ can be
obtained as a limit  of a sequence 
$\{(Y^n,\lambda_0^n,\Lambda_1^n)\}$ corresponding to  
a sequence of bounded Markov process generators $\{(A^n,B^n)\}$ 
(with jump rates going to infinity, if $A$, $B$ are not bounded) 
that approximates $(A,B)$. This construction is carried out 
rigorously in Theorem 2.2 of \cite{Kur91} and yields a 
quite general method to obtain solutions of the controlled 
martingale problem. In the case when there is a 
corresponding patchwork martingale problem, as 
defined in \cite{Kur90} (see Definition \ref{pwmp}), this 
essentially amounts to constructing a solution of the 
patchwork martingale problem, which will be a solution 
of the controlled martingale problem as well: This 
approach is followed in Section \ref{reflect}. See also 
Section \ref{Wentz} for an example of 
another construction by approximation. 
\end{remark}

\begin{remark}\label{stsp}
Note that the requirement that $\lambda_0(t)+\lambda_1(t)=t$ implies any 
solution of the controlled martingale problem for 
$(A,E_0,B,\Xi )$ must satisfy $Y\in D$$_{\bar {E}_0\cup F_1}[0,\infty 
)$. In fact, 
if $Y(t)\in\big(\bar {E}_0\cup F_1\big)^c$ for some $t$, necessarily $
Y(s)\in\big(\bar {E}_0\cup F_1\big)^c$ for all 
$s\in [t,t')$ for some $t'>t$. Then  
$\lambda_0(t')-\lambda_0(t)=\lambda_1(t')-\lambda_1(t)=0$, because $
\lambda_1$ increases only when $Y\in F_1$, 
by $(\ref{Lam1supp})$, and $\lambda_0$ increases only 
when $Y\in\bar {E}_0$, and this 
contradicts $t'-t=\big(\lambda_0(t')-\lambda_0(t))\big)+\big(\lambda_
1(t')-\lambda_1(t)\big)$.
\end{remark}

\begin{remark}
If $(Y,\lambda_0,\Lambda_1)$ is a solution of the 
controlled martingale problem for $(A,E_0,B,\Xi )$ with 
distribution $P$, the canonical process on 
$D_E[0,\infty )\times C_{[0,\infty )}[0,\infty )\times {\cal L}_U$  under $
P$ is also obviously a 
solution with respect to the filtration $\{{\cal B}_t\}$ generated by 
itself. As mentioned in Section \ref{notation}, we denote 
the canonical process under $P$ by $(Y,\lambda_0,\Lambda_1)$ as well. 
\end{remark}

\begin{remark}
One can always assume, without loss of generality, that  
$\{{\cal F}_t\}$ is complete. Then, denoting by $\{{\cal F}_t^Y\}$ the 
smallest complete and right continuous filtration to 
which $Y$ is adapted, 
$\lambda_0$ and $\Lambda_1$  can be replaced by their dual 
predictable projections on $\{{\cal F}_t^Y\}$ so that (\ref{mgp1})  is a 
$\{{\cal F}_t^Y\}$-martingale for each $f\in {\cal D}$ (see Lemma 6.1, \cite{KS01}). 
\end{remark}

\begin{remark}
Note that the controlled martingale problem can also be 
formulated by setting
\[Cf(y,u,v)=vAf(y)+(1-v)Bf(y,u)\]
with controls $(u,v)\in U\times [0,1]$.  The analog of $\Xi$ is 
$\Xi_0\subset E\times U\times [0,1]$ such that
\begin{eqnarray*}
\Xi_0\cap E_0\times U\times [0,1]&=&E_0\times U\times \{1\}\\
\Xi_0\cap\partial E_0\times U\times [0,1]&=&(\partial E_0\times U
\cap\Xi )\times [0,1]|\cup\partial E_0\times U\times \{1\}\\
\Xi_0\cap\bar {E}^c_0\times U\times [0,1]&=&(\bar {E}^c_0\times U
\cap\Xi )\times \{0\}.\end{eqnarray*}
Then $(Y,\mu )$, with $Y\in D_E[0,\infty )$ and $\mu$ a ${\cal P}
(U\times [0,1])$-valued 
process 
is a solution of the controlled martingale for 
$(C,\Xi_0)$ if there exists a filtration $\{{\cal F}_t\}$ such that $
(Y,\mu )$ is 
$\{{\cal F}_t\}$-adapted and
\[f(Y(t))-f(Y(0))-\int_0^t\int_{U\times [0,1]}Cf(Y(s),u,v)\mu_s(d
u\times dv)ds\]
is an $\{{\cal F}_t\}$-martingale.  Every solution of the controlled 
martingale problem for $(C,\Xi_0)$ gives a solution for the 
controlled martingale problem for $(A,E_0,B,\Xi )$ by defining
\[\lambda_0(t)=\int_0^t\int_{U\times [0,1]}v\mu_s(du\times dv)ds\]
and 
\[\Lambda_1(D)=\int_0^{\infty}\int_{U\times [0,1]}(1-v){\bf  1}_D(
s,u)\mu_s(du\times dv)ds.\]
Conversely, every solution of the controlled martingale 
problem for $(A,E_0,B,\Xi )$ gives a solution of the controlled 
martingale problem for $(C,\Xi_0)$.
\end{remark}

\begin{definition}
We define $\Pi\subset {\cal P}(D_E[0,\infty )\times C_{[0,\infty 
)}[0,\infty )\times {\cal L}_U)$ to be the collection of 
the distributions of solutions of the controlled 
martingale problem for $(A,E_0,B,\Xi )$, and for $\nu\in {\cal P}
(E)$, 
$\Pi_{\nu}\subset\Pi$ to be the collection of distributions such that $
Y(0)$ 
has distribution $\nu$.  

${\cal P}_0$ denotes the collection of $\nu\in {\cal P}(\bar {E}_
0\cup F_1)$ such that 
$\Pi_{\nu}\neq\emptyset$.
\end{definition}

\begin{lemma}\label{picompact}
If ${\cal D}$ is dense in $C(E)$, then the collection of distributions 
of solutions $(Y,\lambda_0,\Lambda_1)$ of 
the controlled martingale problem is compact in 
${\cal P}(D_E[0,\infty )\times C_{[0,\infty )}[0,\infty )\times {\cal L}_
U)$  in the sense of weak convergence 
 (taking the Skorohod topology on $D_E[0,\infty )$ 
and the compact uniform topology on $C_{[0,\infty )}[0,\infty )$).  
Consequently, $\Pi$ and $\Pi_{\nu}$, $\nu\in {\cal P}_0$, are compact and 
convex.
\end{lemma}

\begin{proof}
Relative compactness for the family of $Y$ follows from 
Theorems 3.9.4 and 3.9.1 of \cite{EK86}.  The relative 
compactness of the $\lambda_0$ and $\Lambda_1$ is immediate, as $
\lambda_0$ and $\lambda_1$ 
are Lipschitz continuous with Lipschitz constant $1$. The 
fact that every limit point is a solution of the 
controlled martingale problem follows by standard 
arguments from the properties of weakly converging measures 
and from uniform 
integrability of the martingales in (\ref{mgp1}). 

Convexity is immediate.
\end{proof}

\subsection{Closure properties of $\Pi$ }

\begin{lemma}\label{abscon}
Let $(Y,\lambda_0,\Lambda_1)$ be a solution of the controlled
martingale problem for $(A,E_0,B,\!\Xi )$ with filtration $\{{\cal F}_
t\}$.  
Let $H\geq 0$ be a ${\cal F}_0$-measurable random variable such that 
$E[H]=1$. 
Then $P^{H}\in {\cal P}(D_E[0,\infty )\times C_{[0,\infty )}[0,
\infty )\times {\cal L}_U)$ defined by 
\[P^{H}(C)\equiv E[H\,{\bf 1}_C(Y,\lambda_0,\Lambda_1)],\quad C
\in {\cal B}(D_E[0,\infty )\times C_{[0,\infty )}[0,\infty )\times 
{\cal L}_U),\]
is in $\Pi$.
\end{lemma}

\begin{proof}
If $M$ is a $\{{\cal F}_t\}$-martingale under $P$ and $|M(t)|\leq 
C(1+t)$ for some 
$C>0$, then $M$ is a $\{{\cal F}_t\}$-martingale under $P^H$. 
\end{proof}

\begin{lemma}\label{clin}
\item[a)]If $\nu_1<<\nu_2$ and $\Pi_{\nu_2}\neq\emptyset$, then $
\Pi_{\nu_1}\neq\emptyset$.

\item[b)]There exists a closed $F_2\subset\bar {E}_0\cup F_1$ such that 
${\cal P}_0={\cal P}(F_2)$.
\end{lemma}

\begin{proof}
Taking $H=\frac {d\nu_1}{d\nu_2}$, part (a) follows from Lemma \ref{abscon}.

Suppose $P\in\Pi_{\nu}$ and $z\in\mbox{\rm supp}(\nu )$.  Then for each $
\epsilon >0$, 
$\nu (B_{\epsilon}(z))>0$, and setting $H_{\epsilon}(x)=\frac 1{\nu 
(B_{\epsilon}(z))}{\bf 1}_{B_{\epsilon}(z)}(x)$, by 
Lemma \ref{abscon}, $P^{H_{\epsilon}}\in\Pi$. By the compactness of $
\Pi$, 
$P^{H_{\epsilon}}$ will have at least one limit point $P_z$ as $\epsilon
\rightarrow 0$, and 
$P_z\in\Pi_{\delta_z}$.  

Let $F_2$ be the closure of $\cup_{\nu\in {\cal P}_0}\mbox{\rm supp}
(\nu )$.  
Then for each $x\in F_2$, $\Pi_{\delta_x}\neq\emptyset$, and by convexity, for 
$\nu_{x,p}=\sum_{i=1}^mp_i\delta_{x^i}$, $x^i\in F_2$, $p_i\geq 0$, $
\sum_{i=1}^mp_i=1$, $\Pi_{\nu_{x,p}}\neq\emptyset$.  Since 
every $\nu\in {\cal P}(F_2)$ can be approximated by probability 
measures of this form, $\Pi_{\nu}\neq\emptyset$ for each $\nu\in 
{\cal P}(F_2)$.
\end{proof}

\begin{lemma}\label{ctrans}
Define $Y^{\tau}$$,\lambda_0^{\tau}$, and $\Lambda_1^{\tau}$ by 
\begin{equation}\begin{array}{c}
Y^{\tau}(t)=Y(\tau +t),\quad\lambda_0^{\tau}(t)=\lambda_0(\tau +t
)-\lambda_0(\tau ),\quad t\geq 0,\\
\Lambda_1^{\tau}([0,t]\times C)=\Lambda_1([\tau ,\tau +t]\times C
),\quad t\geq 0,C\in {\cal B}(U).\end{array}
\label{shift}\end{equation}
Note that $Y^{\tau}$$,\lambda_0^{\tau}$, and $\Lambda_1^{\tau}$ 
are adapted to the filtration $\{{\cal F}_{\tau +t}\}$.

Then the measure 
$P^{\tau ,H}\in {\cal P}(D_E[0,\infty )\times C_{[0,\infty )}[0,\infty 
)\times {\cal L}_U)$ defined by 
\begin{equation}P^{\tau ,H}(C)=E[H{\bf 1}_C(Y^{\tau},\lambda_0^{\tau}
,\Lambda_1^{\tau})],\quad C\in {\cal B}(D_E[0,\infty )\times C_{[
0,\infty )}[0,\infty )\times {\cal L}_U)\label{cphdist}\end{equation}
is the distribution of 
a solution of the controlled martingale problem for 
$(A,E_0,B,\Xi )$.
\end{lemma}

\begin{proof}
For  $0\leq t<t+r$ and $C\in {\cal B}_t$
\begin{eqnarray*}
&&E^{P^{\tau ,H}}\Big[\Big\{f(Y(t+r))-f(Y(t))-\int_t^{t+r}Af(Y(s)
)d\lambda_0(s)\\
&&\qquad -\int_{(t,t+r]\times U}Bf(Y(s),u)\Lambda_1(ds\times du)\Big
\}{\bf 1}_C(Y,\lambda_0,\Lambda_1)\Big]\\
&&\quad =E\Big[\Big\{f(Y^{\tau}(t+r))-f(Y^{\tau}(t))-\int_t^{t+r}
Af(Y^{\tau}(s))d\lambda^{\tau}_0(s)\\
&&\qquad -\int_{(t,t+r]\times U}Bf(Y^{\tau}(s),u)\Lambda^{\tau}_1
(ds\times du)\Big\}H{\bf 1}_C(Y^{\tau},\lambda_0^{\tau},\Lambda_1^{
\tau})\Big]\\
&&\quad =0\end{eqnarray*}
by the optional sampling theorem.  Therefore,  
$P^{\tau ,H}\in\Pi$.
\end{proof}

\begin{lemma}\label{paste}
Suppose that $(Y,\lambda_0,\Lambda_1)$ is a solution of the controlled 
martingale problem with filtration $\{{\cal F}_t\}$ and that $\tau$ is a 
finite $\{{\cal F}_t\}$-stopping time.  Let $P^0\in {\cal P}(D_E[
0,\infty )\times C_{[0,\infty )}[0,\infty )\times {\cal L}_U\times 
[0,\infty ))$ 
be the joint distribution of the 4-tuple of random 
variables $(Y,\lambda_0,\Lambda_1,\tau )$.  Let $\nu$ be the distribution of 
$Y(\tau )$, and let $P^1\in\Pi_{\nu}$ (not empty by Lemma \ref{ctrans}).  Then there exists 
$P\in {\cal P}(D_E[0,\infty )\times C_{[0,\infty )}[0,\infty )\times 
{\cal L}_U\times [0,\infty ))$ and a filtration 
$\{{\cal H}_t\}$ in $D_E[0,\infty )\times C_{[0,\infty )}[0,\infty 
)\times {\cal L}_U\times [0,\infty )$ 
such that, under $P$, $(Y,\lambda_0,\Lambda_1)$ is a solution of the 
controlled martingale problem with filtration $\{{\cal H}_t\}$, 
$\tau$ is a $\{{\cal H}_t\}$-stopping time, 
$(Y(\cdot\wedge\tau ),\lambda_0(\cdot\wedge\tau ),\Lambda_1(\cdot
\wedge\tau ,\cdot ),\tau )$  has the same distribution 
under $P^0$ and $P$ and the distribution of $(Y^{\tau},\lambda_0^{
\tau},\Lambda_1^{\tau})$ under 
$P$ is $P^1$.
\end{lemma}

\begin{proof}
Let
\[\Omega =D_E[0,\infty )\times C_{[0,\infty )}[0,\infty )\times {\cal L}_
U\times [0,\infty )\times D_E[0,\infty )\times C_{[0,\infty )}[0,
\infty )\times {\cal L}_U,\]
and denote the elements by $(Y^0,\lambda_0^0,\Lambda_1^0,\tau^0,Y^
1,\lambda_0^1,\Lambda_1^1)$.  Apply 
Lemma 4.5.15 of  \cite{EK86}  to $P^0$ and $P^1$ to obtain $P$ 
on $\Omega$ such that $Y^0(\tau )=Y^1(0)$  and define 
\begin{eqnarray*}
Y(t)&=&\left\{\begin{array}{lc}
Y^0(t),&\quad t<\tau^0\\
Y^1(t-\tau^0),&\quad t\geq\tau^0\end{array}
\right.\\
\lambda_0(t)&=&\left\{\begin{array}{lc}
\lambda^0_0(t),&\quad t<\tau^0\\
\lambda_0^0(\tau^0)+\lambda^1_0(t-\tau^0),&\quad t\geq\tau^0\end{array}
\right.\\
\Lambda_1([0,t]\times C)&=&\left\{\begin{array}{lc}
\Lambda_1^0([0,t]\times C),&\quad t<\tau^0\\
\Lambda_1^0([0,\tau^0]\times C)+\Lambda^1_1([0,t-\tau^0]\times C)
,&\quad t\geq\tau^0.\end{array}
\right.\end{eqnarray*}
The fact that $(Y,\lambda_0,\Lambda_1)$ is a solution of the controlled 
martingale problem follows as in the proof of Lemma 
4.5.16 of \cite{EK86}.
\end{proof}

\setcounter{equation}{0}

\section{Constrained martingale problems}\label{CMP}

As discussed in the Introduction and at the beginning of 
Section \ref{ClMP}, we are interested in processes that in $E_0$ behave like 
solutions of the martingale problem for the operator $A$, 
are constrained to remain in $\bar {E}_0$, 
and whose behavior on $\partial E_0$ is determined by the 
operator $B$. In Section \ref{ClMP}, we have introduced a 
controlled process $Y$ with values in all of $E$, that 
evolves on a slower time scale and whose behavior in 
$E_0^c$ is determined by $B$. $Y$ is the first element of a triple 
$(Y,\lambda_0,\Lambda_1)$ that is a solution of the controlled martingale 
problem (Definition \ref{clmp}). 
We now construct the constrained process, $X,$  
by time changing $Y$, where the time change is obtained by 
inverting $\lambda_0$. 
The following lemma gives conditions that ensure that 
the process obtained by inverting $\lambda_0$ is defined for all 
time. 

\begin{lemma}\label{invert}
Let $(Y,\lambda_0,\Lambda_1)$ be a solution of the controlled martingale 
problem for $(A,E_0,B,\Xi )$, and define 
\begin{equation}\tau (t)=\inf\{s:\lambda_0(s)>t\},\qquad t\geq 0.\label{tau}\end{equation}
Suppose there is an $f\in {\cal D}$ and $\epsilon >0$ such that
\begin{equation}\int_{[0,t]\times U}Bf(Y(s),u)\Lambda_1(ds,du)\geq
\epsilon\lambda_1(t).\label{invertf}\end{equation}
Then $\lim_{t\rightarrow\infty}\lambda_0(t)=\infty$ almost surely and $
E[\tau (t)]<\infty$, for all $t\ge 0$.
\end{lemma}

\begin{proof}
See Lemma 2.9 of \cite{Kur91}.
\end{proof}

\begin{remark}
$(\ref{invertf})$ is a natural condition which is also used in 
the study of PDEs (see, e.g. \cite{CIL92}, Lemma 7.6). 
An example where it is satisfied is a reflecting diffusion in a smooth domain with a 
nontangential direction of reflection. 
More precisely, let $E_0\equiv \{x:\psi (x)>0\}$ for some 
function $\psi\in C^2(\mathbb R^d)$ such that 
$\psi (x)=0$ implies $\nabla\psi (x)\neq 0$, so that, in particular, the unit 
inward normal at $x\in\partial E_0$ is given by $n(x)\equiv\frac {
\nabla\psi (x)}{|\nabla\psi (x)|}$. Let 
$g:\R^d\rightarrow\R^d$ be a continuous vector field, of unit length on 
$\partial E_0$, such that $\langle g(x),n(x)\rangle >0$ at every $
x\in\partial E_0$. 
Consider the controlled martingale problem 
for $(A,E_0,B,\Xi )$, where 
\begin{eqnarray*}
&Af(x)\equiv\langle\nabla f(x),b(x)\rangle +\frac 12\mbox{\rm tr}\big
(\sigma (x)\sigma^T(x)D^2f(x)\big),\\
&U\equiv \{u\in\R^d:\,|u|=1\},\quad\Xi\equiv \{(x,u):\,x\in\partial 
E_0,\,u=g(x)\},\\
&Bf(x,u)\equiv\langle\nabla f(x),u\rangle ,
\end{eqnarray*}
and ${\cal D}\equiv C^2_c(\mathbb R^d)$. Then $\psi$ itself satisfies 
$(\ref{invertf})$ (recall that $\partial E_0$ is compact). 
\end{remark}

\begin{lemma}\label{laminf}
Under the assumptions of Lemma \ref{picompact}, if, for 
each $P\in\Pi$, $P\{\tau (0)<\infty \}=1$, then for each $P\in\Pi$, 
$\lim_{t\rightarrow\infty}\lambda_0(t)=\infty$ a.s..
\end{lemma}

\begin{proof}
Let $(Y,\lambda_0,\Lambda_1)$ have distribution in $\Pi$.
Then by Lemma \ref{ctrans} and the 
compactness of $\Pi$ (Lemma \ref{picompact}), there exists $t_n\rightarrow
\infty$ such that 
$(Y^{t_n},\lambda_0^{t_n},\Lambda_1^{t_n})\Rightarrow (Y^{\infty}
,\lambda_0^{\infty},\Lambda_1^{\infty})$.  But 
\[P\{\lim_{t\rightarrow\infty}\lambda_0(t)<\infty \}\leq P\{\lim_{
n\rightarrow\infty}\lim_{t\rightarrow\infty}\lambda_0(t_n+t)-\lambda_
0(t_n)=0\}\leq P\{\sup_t\lambda_0^{\infty}(t)=0\}.\]
Since by assumption, $P\{\tau^{\infty}(0)<\infty \}=1$, 
$P\{\sup_t\lambda_0^{\infty}(t)=0\}=0$.
\end{proof}

\begin{lemma}\label{increasing} 
Suppose every solution of the controlled martingale 
problem for $(A,E_0,B,\Xi )$ satisfies $\lambda_0(t)>0$ for all $
t>0$, 
a.s. (i.e. $P\{\tau (0)=0\}=1$ for each $P\in\Pi$). 
Then, for every solution, $\lambda_0$ is a.s.  strictly 
increasing. 
\end{lemma}

\begin{proof}
For each $s>0$, for every solution $(Y,\lambda_0,\Lambda_1)$ of the 
controlled martingale problem for $(A,E_0,B,\Xi )$, with the 
notation of Lemma \ref{ctrans}
$(Y^s,\lambda_0^s,\Lambda_1^s)$ is also a solution. 
\end{proof}

With Lemmas \ref{picompact}, \ref{invert}, 
\ref{increasing} and \ref{laminf} in mind, 
throughout the remainder of the paper, we assume the 
following:  

\begin{condition}\label{cond1}
\item[a)]${\cal D}$ is dense in $C(E)$.

\item[b)]For each $\nu\in {\cal P}(\bar {E}_0)$, $\Pi_{\nu}\neq\emptyset$ (hence $
F_2\supset\bar {E}_0$, where 
$F_2$ is defined in Lemma \ref{clin}).

\item[c)]For each solution $(Y,\lambda_0,\Lambda_1)$ of the controlled martingale 
problem for $(A,E_0,B,\Xi )$, \break 
$\lim_{t\rightarrow\infty}\lambda_0
(t)=\infty$ almost surely.
\end{condition}

\begin{theorem}\label{const} 
Let $(Y,\lambda_0,\Lambda_1)$ be a solution 
of the controlled martingale problem for $(A,E_0,B,\Xi )$ with 
right continuous filtration $\{{\cal F}_t\}$. Let
$\tau (t)$ be given by $(\ref{tau})$, and define ${\cal G}_t={\cal F}_{
\tau (t)}$.  Define 
\[X(t)\equiv Y(\tau (t))\]
and 
\[\Lambda ([0,t]\times C)\equiv\int_{[0,\tau (t)]\times U}{\bf 1}_
C(Y(s),u)\Lambda_1(ds\times du),\quad C\in {\cal B}(\Xi ).\]
Suppose there exists a sequence $\eta_n$ of $\{{\cal G}_t\}$-stopping 
times such that $\eta_n\rightarrow\infty$ and, for each $n$, $E[\tau 
(\eta_n)]<\infty$.

Then $X\in D_{\bar {E}_0}[0,\infty )$, and, for each $f\in {\cal D}$,
\begin{equation}f(X(t))-f(X(0))-\int_0^tAf(X(s))ds-\int_{[0,t]\times
\Xi}Bf(x,u)\Lambda (ds\times dx\times du)\label{mgp2}\end{equation}
is a $\{{\cal G}_t\}$-local martingale.
\end{theorem}

\begin{proof}
Since $\tau (t)$ must be a point of increase of $\lambda_0$, $Y(\tau 
(t))$ must 
be in $\bar {E}_0$.  Since $Y$ and $\tau$ are right continuous, $
X$ must be 
in $D_{\bar {E}_0}[0,\infty ).$

Since
\[\bigg|\int_{[0,t\wedge\eta_n]\times\Xi}Bf(x,u)\Lambda (ds\times 
dx\times du)\bigg|\leq\|Bf\|\lambda_1(\tau (t\wedge\eta_n))\leq\|
Bf\|\tau (t\wedge\eta_n),\]
(\ref{mgp2}) stopped at $\eta_n$ is a martingale.
\end{proof}

\begin{remark}\label{init}
If $\lambda_0(t)>0$ for all $t>0$, in particular if  $Y(0)\in E_0$, then 
$X(0)=Y(0)$, but if $\lambda_0(t)=0$ for some $t>0$, then 
$\tau (0)>0$, and $X(0)=Y(\tau (0))$ may not be $Y(0)$. 

Let ${\cal Q}_0$ be the collection of $\nu\in {\cal P}(\bar {E}_0
)$ such that 
$\nu ={\cal L}(X(0))={\cal L}(Y(\tau (0)))$, for some solution $(
Y,\lambda_0,\Lambda_1)$ of the 
controlled martingale problem, i.e. ${\cal Q}_0$ is the set of 
possible initial distributions of the process $X$ constructed 
in Theorem \ref{const}. Then, by Lemma 
\ref{ctrans}, ${\cal Q}_0$ is the collection of $\nu\in {\cal P}(
\bar {E}_0)$ 
such that there exists $(Y,\lambda_0,\Lambda_1)$ with initial distribution 
$\nu$ for which $\lambda_0(t)>0$ for all $t>0$ a.s.. Note that 
${\cal Q}_0\supset {\cal P}(E_0)$. 
\end{remark}

\begin{definition}\label{cmgpb}
A process $X$ in $D_{\bar {E}_0}[0,\infty )$ is a solution of the 
constrained (local) martingale problem for $(A,E_0,B,\Xi )$ if 
there exists a random measure $\Lambda$ in ${\cal L}_{\Xi}$ and a filtration 
$\{{\cal G}_t\}$ such that $X$ and $\Lambda$ are $\{{\cal G}_t\}$-adapted
and for each 
$f\in {\cal D}$,
(\ref{mgp2}) is a $\{{\cal G}_t\}$-(local) martingale. We may  
assume, without loss of generality, that $\{{\cal G}_t\}$ is right 
continuous. 

A solution obtained as in Theorem \ref{const} from a 
solution of the controlled martingale problem will be 
called {\em natural}.  $\Gamma\subset {\cal P}(D_{\bar {E}_0}[0,\infty 
))$ will denote the set of 
distributions of natural solutions and, for $\nu\in {\cal P}(\bar {
E}_0)$, 
$\Gamma_{\nu}$ will denote the set of distributions of 
natural solutions $X$ such that $X(0)$ has distribution $\nu$.  
\end{definition}

\begin{corollary}\label{const-exist}
\item[a)] For $\nu\in {\cal Q}_0$ $({\cal Q}_0$ defined in Remark 
\ref{init}), if there exists a solution $(Y,\lambda_0,\Lambda_1)$ of the 
controlled martingale problem for $(A,E_0,B,\Xi )$ with initial 
distribution $\nu$ that satisfies the conditions of 
Lemma \ref{invert}, then there exists a natural solution 
to the constrained martingale 
problem for $(A,E_0,B,\Xi )$ with initial distribution $\nu$.  

\item[b)] For $\nu\in {\cal P}(\bar {E}_0)$, if there exists a 
solution $(Y,\lambda_0,\Lambda_1)$ of the controlled 
martingale problem for $(A,E_0,B,\Xi )$ with initial 
distribution $\nu$ such that $\lambda_0$ is strictly increasing a.s. 
(see Lemma \ref{increasing} for a sufficient condition), then there exists a natural 
solution to the constrained local martingale 
problem for $(A,E_0,B,\Xi )$ with initial distribution $\nu$. 
\end{corollary}

\begin{proof}
\begin{itemize}
\item[a)] Under the conditions of Lemma \ref{invert}, we can take 
$\eta_n=n$ and $(\ref{mgp2})$ is actually a martingale. 

\item[b)] If $\lambda_0$ is strictly increasing, then $\tau$ is continuous 
and we can take $\eta_n=\inf\{t:\tau (t)>n\}$.  
\end{itemize}
\end{proof}

We conclude this section with a result giving conditions 
that imply a solution of the constrained martingale problem is natural.

\begin{proposition}\label{nat-char} 
Suppose that $X$ is a solution of the constrained 
martingale problem for $(A,E_0,B,\Xi )$ and $\Lambda$ is the associated 
random measure. If $\Lambda ([0,\cdot ]\times\Xi )$ is continuous and 
for all $h\in C(\Xi )$ and $t>0$,
\begin{equation}\int_{[0,t]\times\Xi}h(x,u)\Lambda (ds\times dx\times 
du)=\int_{[0,t]\times\Xi}h(X(s),u)\Lambda (ds\times dx\times du),\label{concent}\end{equation}
then $X$ is natural. 
\end{proposition}

\begin{proof}
Define 
\[\lambda_0(t)\equiv\inf\{s:s+\Lambda ([0,s]\times\Xi )>t\},\quad 
Y(t)\equiv X(\lambda_0(t))\]
and
\[\Lambda_1([0,t]\times C)\equiv\int_{[0,\lambda_0(t)]\times\Xi}{\bf 1}_
C(u)\,\Lambda (ds\times dx\times du),\quad C\in {\cal B}(U).\]
Then 
\begin{eqnarray*}
&&f(X(\lambda_0(t)))-f(X(0))-\int_0^{\lambda_0(t)}Af(X(s))ds-\int_{
[0,\lambda_0(t)]\times\Xi}Bf(x,u)\Lambda (ds\times dx\times du)\\
&&\qquad =f(Y(t))-f(X(0))-\int_0^tAf(Y(s))d\lambda_0(s)-\int_{[0,
\lambda_0(t)]\times\Xi}Bf(X(s),u)\Lambda (ds\times dx\times du)\\
&&\qquad =f(Y(t))-f(X(0))-\int_0^tAf(Y(s))d\lambda_0(s)-\int_{[0,
t]\times U}Bf(Y(s),u)\Lambda_1(ds\times du).\end{eqnarray*}

\end{proof}

\setcounter{equation}{0}
\section{The Markov selection theorem}\label{selection}

Our strategy for obtaining a Markov solution for the 
constrained martingale problem for $(A,E_0,B,\Xi )$ generally 
follows the approach in Section 4.5 of \cite{EK86} (which 
in turn is based on an unpublished paper \cite{GG77}).  
With reference to these results, for $h\in C(\bar {E}_0)$, and 
$\nu\in {\cal P}(F_2)$ ($F_2$ defined in Lemma \ref{clin}), define 
\begin{equation}\gamma (\Pi_{\nu},h)\equiv\sup_{P\in\Pi_{\nu}}E^P
[\int_0^{\infty}e^{-\lambda_0(s)}h(Y(s))d\lambda_0(s)].\label{gam}\end{equation}
Recalling that $\Pi_{\nu}$ is compact (Lemma \ref{picompact}), 
we see that the supremum is achieved.  

\begin{lemma}\label{uexist}
For $h\in C(\bar {E}_0)$, there exists $v_h\in B(F_2)$  such that
\[\gamma (\Pi_{\nu},h)=\int_{F_2}v_h(x)\nu (dx),\quad\forall\nu\in 
{\cal P}(F_2),\]
and $v_h$ is upper semicontinuous.
\end{lemma}

\begin{proof}
Suppose first that $h$ is nonnegative.  Let $0<\alpha <1$ and 
$\nu ,\mu_1,\mu_2\in {\cal P}(F_2)$.  Suppose $\nu =\alpha\mu_1+(
1-\alpha )\mu_2$.  Then by 
convexity of $\Pi$, 
\begin{eqnarray}
&&\gamma (\Pi_{\nu},h)\label{gamineq}\\
&&\geq\sup_{P_1\in\Pi_{\mu_1},P_2\in\Pi_{\mu_2}}\bigg\{\alpha E^{
P_1}[\int_0^{\infty}e^{-\lambda_0(s)}h(Y(s))d\lambda_0(s)]\non+(1
-\alpha )E^{P_2}[\int_0^{\infty}e^{-\lambda_0(s)}h(Y(s))d\lambda_
0(s)]\bigg\}\non\\
&&=\alpha\gamma (\Pi_{\mu_1},h)+(1-\alpha )\gamma (\Pi_{\mu_2},h)
.\non\end{eqnarray}
But $\mu_1$ and $\mu_2$ are absolutely continuous with respect to 
$\nu$, so setting $H_i=\frac {d\mu_i}{d\nu}$, by Lemma \ref{abscon}, for $
P\in\Pi_{\nu}$,
\begin{eqnarray*}
&&E^P[\int_0^{\infty}e^{-\lambda_0(s)}h(Y(s))d\lambda_0(s)]\\
&&\qquad =\alpha E^{P^{H_1}}[\int_0^{\infty}e^{-\lambda_0(s)}h(Y(
s))d\lambda_0(s)]+(1-\alpha )E^{P^{H_2}}[\int_0^{\infty}e^{-\lambda_
0(s)}h(Y(s))d\lambda_0(s)]\end{eqnarray*}
so the reverse of the previous inequality holds and hence 
\begin{equation}\gamma (\Pi_{\nu},h)=\alpha\gamma (\Pi_{\mu_1},h)+(1-\alpha )\gamma 
(\Pi_{\mu_2},h).\label{convex}\end{equation}
The compactness of $\Pi$ and the continuity of 
$(Y,\lambda_0,\Lambda_1)\rightarrow\int_0^{\infty}e^{-\lambda_0(s
)}h(Y(s))d\lambda_0(s)$ ensure that the mapping 
$\nu\rightarrow\gamma (\Pi_{\nu},h)$ is upper semicontinuous, and the lemma 
follows by Lemma 4.5.9 of \cite{EK86}. 

If $h$ is not nonnegative, take $v_h\equiv v_{h-\inf h}+\inf h$. 
\end{proof}

\begin{lemma}\label{pih}
Let $\Pi_{\nu}^h\subset\Pi_{\nu}$ be the subset for which the supremum in 
(\ref{gam}) is achieved, that is,  $Q\in\Pi_{\nu}^h$ if and only if 
\[E^Q[\int_0^{\infty}e^{-\lambda_0(s)}h(Y(s))d\lambda_0(s)]=\gamma 
(\Pi_{\nu},h)=\int_{F_2}v_h(x)\nu (dx).\]
Defining $\Pi^h=\cup_{\nu\in {\cal P}(F_2)}\Pi^h_{\nu}$, $\Pi^h$ is convex, and for each $
\nu$, 
$\Pi_{\nu}^h$ is compact (however, it is not clear whether or not $
\Pi^h$ is 
compact).
\end{lemma}

\begin{proof}
Let $P_1\in\Pi^h_{\mu_1}$, $P_2\in\Pi^h_{\mu_1}$ and $P=\alpha P_
1+(1-\alpha )P_2$, $0<\alpha <1$.  
Setting $\nu =\alpha\mu_1+(1-\alpha )\mu_2$, 
\[E^P[\int_0^{\infty}e^{-\lambda_0(s)}h(Y(s))d\lambda_0(s)]\\
=\alpha\gamma (\Pi_{\mu_1},h)+(1-\alpha )\gamma (\Pi_{\mu_2},h)=\gamma 
(\Pi_{\nu},h),\]
where the last equality follows from $(\ref{convex})$. 

Compactness of $\Pi_{\nu}^h$ follows from the compactness of $\Pi_{
\nu}$ 
and the continuity of the functional $(Y,\lambda_0)\rightarrow\int_0^{\infty}e^{
-\lambda_0(s)}h(Y(s))d\lambda_0(s)$.
\end{proof}

Now consider $\{h_n\}\subset C(\bar {E}_0)$, $h_n\ge 0$, and define $
\Pi_{\nu}^{h_1,h_2}$ to 
be the subset of distributions $Q\in\Pi_{\nu}^{h_1}$ such that
\[E^Q[\int_0^{\infty}e^{-\lambda_0(s)}h_2(Y(s))d\lambda_0(s)]=\gamma 
(\Pi^{h_1}_{\nu},h_2)\equiv\sup_{P\in\Pi_{\nu}^{h_1}}E^P[\int_0^{
\infty}e^{-\lambda_0(s)}h_2(Y(s))d\lambda_0(s)],\]
and recursively, define $\Pi_{\nu}^{h_1,\ldots ,h_{n+1}}$ to be the subset of 
distributions $Q\in\Pi_{\nu}^{h_1,\ldots ,h_n}$ such that 
\begin{eqnarray*}
E^Q[\int_0^{\infty}e^{-\lambda_0(s)}h_{n+1}(Y(s))d\lambda_0(s)]&=&
\gamma (\Pi^{h_1,\ldots ,h_n}_{\nu},h_{n+1})\\
&\equiv&\sup_{P\in\Pi_{\nu}^{h_1,\ldots ,h_n}}E^P[\int_0^{\infty}
e^{-\lambda_0(s)}h_{h+1}(Y(s))d\lambda_0(s)]\end{eqnarray*}

Inductively, the compactness of $\Pi^{h_1,\ldots ,h_n}_{\nu}$ and the continuity of  \break 
$(Y,\lambda_0)\rightarrow\int_0^{\infty}e^{-\lambda_0(s)}h_{n+1}(
Y(s))d\lambda_0(s)$
ensure that $\Pi_{\nu}^{h_1,\ldots ,h_{n+1}}$ is compact 
and nonempty. Let $\Pi^{h_1,\ldots ,h_n}$$=\cup_{\nu\in {\cal P}(
F_2)}\Pi_{\nu}^{h_1,\ldots ,h_n}$.

We now need to show the existence of a function 
$v^{h_1,\ldots ,h_n}_{h_{n+1}}$ such that 
\[\gamma (\Pi_{\nu}^{h_1,\ldots h_n},h_{n+1})=\int_{F_2}v_{h_{n+1}}^{
h_1,\ldots ,h_n}(x)\nu (dx),\quad\forall\nu\in {\cal P}(F_2).\]
If $\nu =\alpha\mu_1+(1-\alpha )\mu_2$, then, by the same argument used 
for $(\ref{convex})$, 
\[\gamma (\Pi_{\nu}^{h_1,\ldots h_n},h_{n+1})=\alpha\gamma (\Pi_{
\mu_1}^{h_1,\ldots h_n},h_{n+1})+(1-\alpha )\gamma (\Pi_{\mu_2}^{
h_1,\ldots h_n},h_{n+1});\]
 however, we do not know the upper semicontinuity of 
$\gamma (\Pi_{\nu}^{h_1,\ldots h_n},h_{n+1})$ as a function of $\nu$, because 
it is not clear whether or not $\Pi^{h_1,\ldots h_n}$ is 
compact. Consequently,  we cannot apply 
Lemma 4.5.9 of \cite{EK86} as we did in Lemma 
\ref{uexist}.  

\begin{lemma}\label{uexistn}
For each $n=1,2,\ldots$, $\nu\in {\cal P}(F_2)$, and  $g\in C(\bar {
E}_0)$, there exists 
$v^{h_1,\ldots ,h_n}_g\equiv v^{n+1}_g\in B(F_2)$ such that
\begin{equation}\gamma (\Pi_{\nu}^{h_1,\ldots h_n},g)=\int_{E_2}v^{
n+1}_g(x)\nu (dx).\label{urep}\end{equation}
\end{lemma}

\begin{proof}
Suppose first that $g\geq 0$. 
Following the argument on page 214 of \cite{EK86}, we 
proceed by induction.  For $n=1$, (\ref{urep}) is given 
by Lemma \ref{uexist}.  Assuming (\ref{urep}) holds for $n$,  
we claim
\[v^{n+1}_g(x)\equiv\lim_{\epsilon\rightarrow 0+}\epsilon^{-1}(v^
n_{h_n+\epsilon g}(x)-v^n_{h_n}(x))\]
satisfies $(\ref{urep})$.
Note that for all $\nu\in {\cal P}(F_2)$,
\begin{equation}\int_{F_2}v^n_{h_n+\epsilon g}(x)\nu (dx)\geq\int_{
F_2}v_{h_n}^n(x)\nu (dx)+\epsilon\gamma (\Pi_{\nu}^{h_1,\ldots ,h_
n},g),\label{ineq1}\end{equation}
and hence, for all $x\in F_2$, 
\[\liminf_{\epsilon\rightarrow 0}\epsilon^{-1}(v^n_{h_n+\epsilon 
g}(x)-v^n_{h_n}(x))\geq\gamma (\Pi_{\delta_x}^{h_1,\ldots ,h_n},g
).\]
For each $\epsilon >0$, let $P_{\nu}^{\epsilon}\in\Pi^{h_1,\ldots 
,h_{n-1}}_{\nu}$ satisfy 
\begin{eqnarray}
\int_{F_2}v_{h_n+\epsilon g}^n(x)\nu (dx)&=&E^{P_{\nu}^{\epsilon}}
[\int_0^{\infty}e^{-\lambda_0(s)}(h_n+\epsilon g)(Y(s))d\lambda_0
(s)]]\non\\
&\leq&\int_{F_2}v^n_{h_n}(x)\nu dx)+\epsilon E^{P_{\nu}^{\epsilon}}
[\int_0^{\infty}e^{-\lambda_0(s)}g(Y(s))d\lambda_0(s)]].
\label{ineq2}\end{eqnarray}
By $(\ref{ineq1})$ and $(\ref{ineq2})$, all limit points of $P_{\nu}^{
\epsilon}$ as $\epsilon\rightarrow 0$ are in $\Pi^{h_1,\ldots ,h_
n}_{\nu}$, so
\begin{eqnarray*}
\limsup_{\epsilon\rightarrow 0}\int_{F_2}\epsilon^{-1}(v_{h_n+\epsilon 
g}^n(x)-v^n_{h_n}(x))\nu (dx)&\leq&\limsup_{\epsilon\rightarrow 0}
E^{P^{\epsilon}_{\nu}}[\int_0^{\infty}e^{-\lambda_0(s)}g(Y(s))d\lambda_
0(s)]]\\
&\leq&\gamma (\Pi_{\nu}^{h_1,\ldots ,h_n},g).\end{eqnarray*}
Therefore, 
\[v^{n+1}_g(x)\equiv\lim_{\epsilon\rightarrow 0}\epsilon^{-1}(v_{
h_n+\epsilon g}^n(x)-v^n_{h_n}(x))\]
exists, and since, again by $(\ref{ineq1})$ and $(\ref{ineq2})$, 
\[0\leq\epsilon^{-1}(v_{h_n+\epsilon g}^n(x)-v^n_{h_n}(x))\leq\sup_
zg(z),\]
(\ref{urep}) holds by the dominated convergence theorem.

If $g$ is not nonnegative, take $v^{n+1}_g\equiv v_{g-\inf g}+\inf 
g$. 
\end{proof}

\subsection{Closure properties of $\Pi^{h_1,\ldots ,h_n}$}

\begin{lemma}\label{chginit}
Suppose $(Y,\lambda_0,\Lambda_1)$ is a solution of the controlled 
martingale problem
with filtration $\{{\cal F}_t\}$ and distribution 
$P\in\Pi^{h_1,\ldots ,h_n}$.  Let $H\geq 0$ be ${\cal F}_0$-measurable with $
E[H]=1$. 
Then $P^H$  defined as in Lemma \ref{abscon} is in $\Pi^{h_1,\ldots 
,h_n}$. 
\end{lemma}

\begin{proof}
Let $c>0$, $H^c=\frac {H\wedge c}{E[H\wedge c]}$, and $G^c=\frac {
c-H\wedge c}{E[c-H\wedge c]}$.  Then
\begin{eqnarray*}
E^P[v^n_{h_n}(Y(0))]&=&\frac {E[H\wedge c]}cE^{P^{H^c}}[\int_0^{\infty}
e^{-\lambda_0(s)}h_n(Y(s))d\lambda_0(s)]\\
&&\qquad\qquad +\frac {E[c-H\wedge c]}cE^{P^{G^c}}[\int_0^{\infty}
e^{-\lambda_0(s)}h_n(Y(s))d\lambda_0(s)]\\
&\leq&\frac {E[H\wedge c]}cE^{P^{H^c}}[v^n_{h_n}(Y(0))]+\frac {E[
c-H\wedge c]}cE^{P^{G^c}}[v^n_{h_n}(Y(0))]\\
&=&E^P[v^n_{h_n}(Y(0))],\end{eqnarray*}
and since the inequality is termwise, we must have
\[E^{P^{H^c}}[\int_0^{\infty}e^{-\lambda_0(s)}h_n(Y(s))d\lambda_0
(s)]=E^{P^{H^c}}[v^n_{h_n}(Y(0))].\]
Letting $c\rightarrow\infty$, the monotone convergence theorem implies
\begin{equation}E^{P^H}\begin{array}{c}
\end{array}
[\int_0^{\infty}e^{-\lambda_0(s)}h_n(Y(s))d\lambda_0(s)]=E^{P^H}[
v^n_{h_n}(Y(0))],\label{extPI}\end{equation}
and  $P^H\in\Pi^{h_1,\ldots ,h_n}$.
\end{proof}

\begin{remark}\label{reminit}
Note that $(\mbox{\rm \ref{extPI}})$ implies 
\[E^P[\int_0^{\infty}e^{-\lambda_0(s)}h_n(Y(s))d\lambda_0(s)|{\cal B}_
0]=v^n_{h_n}(Y(0)).\]
In particular 
\[v^n_{h_n}(x)=E^P[\int_0^{\infty}e^{-\lambda_0(s)}h_n(Y(s))d\lambda_
0(s)],\quad P\in\Pi^{h_1,\ldots ,h_n}_{\delta_x},\quad x\in F_2.\]
\end{remark}

\begin{lemma}\label{taushiftn} 
Suppose $(Y,\lambda_0,\Lambda_1)$ is a solution of the controlled 
martingale problem with filtration $\{{\cal F}_t\}$ with distribution 
$P\in\Pi^{h_1,\ldots ,h_n}$.  Let $\tau$ be a finite $\{{\cal F}_
t\}$-stopping time and 
let $H\geq 0$ be ${\cal F}_{\tau}$-measurable with $E[H]=1$. 
Then, for $(Y^{\tau},\lambda_0^{\tau},\Lambda^{\tau}_1)$ defined by (\ref{shift}), 
$P^{\tau ,H}$ defined by (\ref{cphdist}) is in 
$\Pi^{h_1,\ldots ,h_n}$ and $\Pi^{h_1,\ldots ,h_n}$ is closed under the pasting 
operation in Lemma \ref{paste}.  
\end{lemma}

\begin{proof}
Again we proceed by induction. 
By Lemma 2.11, 
\begin{eqnarray}
\gamma (\Pi_{\nu},h_1)&=&E[\int_0^{\infty}e^{-\lambda_0(s)}h_1(Y(
s))d\lambda_0(s)]\non\\
&=&E[\int_0^{\tau}e^{-\lambda_0(s)}h_1(Y(s))d\lambda_0(s)]+E[e^{-
\lambda_0(\tau )}\int_0^{\infty}e^{-\lambda_0^{\tau}(s)}h_1(Y^{\tau}
(s))d\lambda_0^{\tau}(s)]\non\\
&=&E[\int_0^{\tau}e^{-\lambda_0(s)}h_1(Y(s))d\lambda_0(s)]+E[e^{-
\lambda_0(\tau )}]E^{P^{\tau, H_0}}[\int_0^{\infty}e^{-\lambda_0^{\tau}
(s)}h_1(Y^{\tau}(s))d\lambda_0^{\tau}(s)]\non\\
&\leq&E[\int_0^{\tau}e^{-\lambda_0(s)}h_1(Y(s))d\lambda_0(s)]+E[e^{
-\lambda_0(\tau )}]\gamma (\Pi_{\mu},h_1),\label{pasteh}\end{eqnarray}
where 
\[H_0\equiv\frac {e^{-\lambda_0(\tau )}}{E[e^{-\lambda_0(\tau )}]}
,\quad\mu (C)\equiv E[H_0{\bf 1}_C(Y(\tau ))].\]
Let $\zeta$ be the distribution of $Y(\tau )$ and let $P^1\in\Pi_{
\zeta}^{h_1}$.
Taking  $(Y^0,\lambda^0_0,\Lambda^0_1,\tau^0)$ with the same distribution as $
(Y,\lambda_0,\Lambda_1,\tau )$
and $(Y^1,\lambda_0^1,\Lambda_1^1)$ with distribution $P^1$,  let $
(\hat {Y},\hat{\lambda}_0,\hat{\Lambda}_1,\hat{\tau })$ be  
given by Lemma 2.12. Then, for $\hat {H}_0\equiv\frac {e^{-\hat{\lambda}_
0(\hat{\tau })}}{E[e^{-\hat{\lambda}_0(\hat{\tau })}]}$,
\begin{eqnarray*}
&&E[\int_0^{\infty}e^{-\hat{\lambda}_0(s)}h_1(\hat {Y}(s))ds]\\
&&=E[\int_0^{\tau}e^{-\lambda_0(s)}h_1(Y(s))d\lambda_0(s)]+E[e^{-
\hat{\lambda}_0(\hat{\tau })}]E[\hat {H}_0\int_0^{\infty}e^{-\hat{
\lambda}_0^{\hat{\tau}}(s)}h_1(\hat {Y}^{\hat{\tau}}(s))d\hat{\lambda}_
0^{\hat{\tau}}(s)]\\
&&=E[\int_0^{\tau}e^{-\lambda_0(s)}h_1(Y(s))d\lambda_0(s)]+E[e^{-
\hat{\lambda}_0(\hat{\tau })}]E^{P^{\hat {H}_0}}[\int_0^{\infty}e^{
-\hat{\lambda}_0^{\hat{\tau}}(s)}h_1(\hat {Y}^{\hat{\tau}}(s))d\hat{
\lambda}_0^{\hat{\tau}}(s)]\\
&&=E[\int_0^{\tau}e^{-\lambda_0(s)}h_1(Y(s))d\lambda_0(s)]+E[e^{-
\lambda_0(\tau )}]\gamma (\Pi_{\mu},h_1)\\
&&\geq\gamma (\Pi_{\nu},h_1),\end{eqnarray*}
where the third equality holds by Lemma 4.4 and the inequality is given by (\ref{pasteh}).  
Consequently, equality must hold here and in (\ref{pasteh}), 
giving both that $P^{\tau ,H_0}$ is in $\Pi^{h_1}$ 
and that $\Pi^{h_1}$ is closed under the pasting operation. Now 
for an arbitrary $H$ as in the statement of the theorem, 
note that the probability measure $P^H$  can be written 
as 
\[P^H(C)=E^{P^{H_0}}[HH_0^{-1}\mathbf{1}_C],\quad C\in {\cal B}(D_
E[0,\infty )\times C_{[0,\infty )}[0,\infty )\times {\cal L}_U),\]
where $E^{P^{H_0}}[HH_0^{-1}]=1.$ Since $P^{\tau ,H_0}$ is the distribution 
of $(Y^{\tau},\lambda_0^{\tau},\Lambda^{\tau}_1)$ under $P^{H_0}$, Lemma 4.4 yields that 
the distribution of $(Y^{\tau},\lambda_0^{\tau},\Lambda^{\tau}_1)$ under $
P^H$ is in $\Pi^{h_1}$, i.e. 
$P^{\tau ,H}$ is in $\Pi^{h_1}$. 

Now suppose that the result holds for $1\leq k\leq n-1$.
 In particular, if the distribution of $(Y,\lambda_0,\Lambda_1)$ is in 
$\Pi^{h_1,\ldots ,h_{n-1}}$, then the distribution of $(Y^{\tau},
\lambda_0^{\tau},\Lambda^{\tau}_1)$ under 
$P^{H_0}$ is in $\Pi^{h_1,\ldots ,h_{n-1}}$.  With this observation, the proof 
of the result for $n$ follows.
\end{proof}

\subsection {The martingale property and the Markov 
selection theorem}

\begin{lemma}\label{Ymgrel}
Let $(Y,\lambda_0,\Lambda_1)$ be a solution of the controlled martingale 
problem for $(A,E_0,B,\Xi )$ with filtration $\{{\cal F}_t\}$ and distribution in $
\Pi^{h_1,\ldots ,h_n}$.  
For $v^n_{h_n}$ given by Lemma \ref{uexistn}, 
\[e^{-\lambda_0(t)}v^n_{h_n}(Y(t))+\int_0^te^{-\lambda_0(s)}h_n(Y
(s))d\lambda_0(s)\]
is a $\{{\cal F}_t\}$-martingale, and
\begin{equation}v_{h_n}(Y(0))=E[\int_0^{\infty}e^{-\lambda_0(t)}h_
n(Y(s))d\lambda_0(s)|{\cal F}_0].\label{vdef}\end{equation}
\end{lemma}

\begin{proof}
For  $t\geq 0$ and $H$ bounded and ${\cal F}_t$-measurable, by Lemma 
\ref{taushiftn} and Remark \ref{reminit}
\begin{eqnarray*}
E[\int_t^{\infty}e^{-\lambda_0(s)}h_n(Y(s)d\lambda_0(s)H]&=&E[e^{
-\lambda_0(t)}\int_0^{\infty}e^{-\lambda_0^t}h_n(Y^t(s)d\lambda_0^
t(s)H]\\
&=&E[e^{-\lambda_0(t)}v^n_{h_n}(Y(t))H],\end{eqnarray*}
and hence 
\[E[\int_t^{\infty}e^{-\lambda_0(s)}h_n(Y(s)d\lambda_0(s)|{\cal F}_
t]=e^{-\lambda_0(t)}v^n_{h_n}(Y(t))\]
and 
\[E[\int_0^{\infty}e^{-\lambda_0(s)}h_n(Y(s)d\lambda_0(s)|{\cal F}_
t]=e^{-\lambda_0t}v^n_{h_n}(Y(t))+\int_0^te^{-\lambda_0(s)}h_n(Y(
s))d\lambda_0(s).\]
The left side is clearly a martingale, and (\ref{vdef}) 
follows by taking $t=0$.
\end{proof}

Recall that we are assuming Condition \ref{cond1}.  In 
particular, we are assuming that for all solutions of the 
controlled martingale problem, $\lambda_0(t)\rightarrow\infty$. 

\begin{theorem}\label{Xmgrel} 
For $\nu\in {\cal P}(F_2)$, let $\Pi^{\infty}_{\nu}\equiv\cap_n\Pi^{
h_1,\ldots ,h_n}_{\nu}$ (note that $\Pi^{\infty}_{\nu}\neq\emptyset$) 
and $\Pi^{\infty}\equiv\cup_{\nu\in {\cal P}(F_2)}\Pi^{\infty}_{\nu}$. 
Let $(Y,\lambda_0,\Lambda_1)$ be a solution of 
the controlled martingale problem with filtration $\{{\cal F}_t\}$ and 
distribution in $\Pi^{\infty}$. Define, as in Theorem \ref{const}, 
$\tau (t)\equiv\inf\{s:\lambda_0(s)>t\}$ and $X(t)\equiv Y(\tau (
t))$.  Then for all $h_n$, 
\[v^n_{h_n}(X(t))-\int_0^t(v^n_{h_n}(X(s))-h_n(X(s)))ds,\]
is a $\{{\cal F}_{\tau (t)}\}$-martingale. 
\end{theorem}

\begin{proof}
For each $h_n$, by Lemma \ref{Ymgrel} 
\[e^{-\lambda_0(t)}v^n_{h_n}(Y(t))+\int_0^te^{-\lambda_0(s)}h_n(Y
(s))d\lambda_0(s)\]
is a $\{{\cal F}_t\}$-martingale, so the time changed process 
\[e^{-t}v^n_{h_n}(X(t))+\int_0^te^{-s}h_n(X(s))ds\]
is a $\{{\cal F}_{\tau (t)}\}$-martingale.
 Hence by Lemma 4.3.2 in \cite{EK86}, 
\[v^n_{h_n}(X(t))-\int_0^t(v^n_{h_n}(X(s))-h_n(X(s)))ds,\]
is a $\{{\cal F}_{\tau (t)}\}$-martingale.
\end{proof}

Let ${\cal Q}_0^{\infty}$ be the collection of $\nu\in {\cal P}(\bar {
E}_0)$ such that 
$\nu ={\cal L}(X(0))={\cal L}(Y(\tau (0)))$, for some $(Y,\lambda_
0,\Lambda_1)$ with distribution in $\Pi^{\infty}$  
and $\tau$ and $X$ as in Theorem \ref{Xmgrel}. Then, by Lemma 
\ref{taushiftn}, 
${\cal Q}_0^{\infty}$ is the collection of $\nu\in {\cal P}(\bar {
E}_0)$ 
such that there exists $(Y,\lambda_0,\Lambda_1)$ with distribution in $
\Pi^{\infty}_{\nu}$ 
for which $\lambda_0(t)>0$ for all $t>0$ a.s. Note that 
${\cal Q}_0^{\infty}\supset {\cal P}(E_0)$. In particular $\delta_
x\in {\cal Q}_0^{\infty}$ for every $x\in E_0$.

\begin{theorem}\label{uniqueness} 
Let $\{h_n\}\subset C(\bar {E}_0)$ be such that its linear span is dense in $
B(\bar {E}_0)$ 
under bounded pointwise convergence.  For $\nu\in {\cal Q}_0^{\infty}$, let 
$\Gamma^{\infty}_{\nu}$ be the collection of distributions of processes 
$X\equiv Y\circ\tau$ defined as in Theorem \ref{Xmgrel} with 
$\nu ={\cal L}(X(0))$ and $(Y,\lambda_0,\Lambda_1)$ with distribution in $
\Pi^{\infty}$.
Then, there exists one and only one distribution in $\Gamma^{\infty}_{
\nu}$ and it is 
the distribution of a strong Markov process.  
\end{theorem}

\begin{proof}
By Remark \ref{reminit} and Theorem \ref{Xmgrel}, 
for each $n$, $(v_{h_n}^n,h_n)$ is a pair$(v_h,h)$ such that 
\begin{equation}v_h(Y(0))=E^P[\int_0^{\infty}e^{-\lambda_0(s)}h(Y
(s))d\lambda_0(s)|{\cal B}_0],\quad\forall P\in\Pi^{\infty},\label{reso}\end{equation}
and 
\begin{equation}v_h(X(t))-\int_0^t[v_h(X(s))-h(X(s)]ds\label{mgprop}\end{equation}
is a $\{{\cal B}_{\tau (t)}\}$-martingale for each $X=Y\circ\tau$, $
(Y,\lambda_0,\Lambda_1)$ with 
distribution in $\Pi^{\infty}$. 
Let 
\[\mathbb A=\{(\left.v_h\right|_{\bar {E}_0},\left.v_h\right|_{\bar {
E}_0}-h):\mbox{\rm such that }(v_h,h)\in B(F_2)\times B(\bar {E}_
0)\mbox{\rm \ satisfies (\ref{reso}) and (\ref{mgprop}}\}.\]
$\mathbb A$ is linear and closed under bounded pointwise convergence. 

For $(v_h,h)$ such that $(\left.v_h\right|_{\bar {E}_0},\left.v_h\right
|_{\bar {E}_0}-h)\in\mathbb A$, 
by Lemma 4.3.2 of \cite{EK86}, for each $\eta >0$ and 
$X=Y\circ\tau$ as in $(\ref{mgprop})$, 
\[e^{-\eta t}v_h(X(t))+\begin{array}{c}
\end{array}
\int_0^te^{-\eta s}(\eta v_h(X(s))-v_h(X(s))+h(X(s)))ds\]
is a $\{{\cal B}_{\tau (t)}\}$-martingale, and hence
\begin{equation}v_h(X(0))=E[\int_0^{\infty}e^{-\eta s}(\eta v_h(X
(s))-v_h(X(s))+h(X(s)))ds|{\cal B}_{\tau (0)}].\label{diss}\end{equation}
$(\ref{diss})$ with $\eta =1$ and $(\ref{reso})$ imply  
\[v_h(Y(0))=E[v_h(X(0))|{\cal B}_0].\]
Consequently, for each $x\in\bar {E}_0$, for $(Y,\lambda_0,\Lambda_
1)$ with distribution in $\Pi^{\infty}_{\delta_x}$, 
$X=Y\circ\tau$, 
\[v_h(x)=E[\int_0^{\infty}e^{-\eta s}(\eta v_h(X(s))-v_h(X(s))+h(
X(s)))ds],\]
and, as in Proposition 4.3.5 of \cite{EK86}, this implies 
that $\mathbb A$ is dissipative.

Since ${\cal R}(I-\mathbb A)\supset \{h_n\}$ and 
the linear span of $\{h_n\}$ is bounded pointwise dense in 
$B(\bar {E}_0)$, we have $\overline {{\cal R}(I-\mathbb A)}^{bp}=
B(\bar {E}_0)$. 
The properties of resolvents of dissipative 
operators (for example, Lemma 1.2.3 of \cite{EK86}) 
ensure that $\overline {{\cal R}(\eta I-\mathbb A)}^{bp}=B(\bar {
E}_0)$ for all $\eta >0$. 
Therefore, by Corollary 4.4.4 of \cite{EK86}, 
for each $\nu\in {\cal Q}_0^{\infty}$ uniqueness holds for 
the martingale problem for $\mathbb A$ with 
initial distribution $\nu$, and, by construction, 
the distribution of the solution is the 
unique distribution in $\Gamma^{\infty}_{\nu}$. 

Now let $(Y,\lambda_0,\Lambda_1)$ be the canonical process with 
distribution $P\in\Pi^{\infty}$ such that ${\cal L}(Y(\tau (0)))$ $=
\nu$, so that the 
distribution of 
$X\equiv Y\circ\tau$, defined as in Theorem \ref{Xmgrel}, is the 
unique distribution in $\Gamma^{\infty}_{\nu}$. In order to show that $
X$ is a 
strong Markov process we need to show that, for each 
$\{{\cal B}_{\tau (t)}\}$ finite stopping time $\sigma$, 
$\tau (\sigma )$ is a $\{{\cal B}_t\}$-stopping time and, setting $
X^{\sigma}(\cdot )=X(\sigma +\cdot )$, for every $F\in {\cal B}_{
\tau (\sigma )}$, 
\begin{equation}E^P[{\bf 1}_F{\bf 1}_B(X^{\sigma})]=E^P[{\bf 1}_F
E[{\bf 1}_B(X^{\sigma})|X(\sigma )]],\qquad\forall B\in {\cal B}(
{\cal D}_{\bar {E}_0}[0,\infty )).\label{stM}\end{equation}

The fact that $\tau (\sigma )$ is a $\{{\cal B}_t\}$-stopping time follows by the 
right continuity of $\{{\cal B}_t\}$ and the observation that 
\[\{\tau (\sigma )<s\}=\cup_{t\in {\mathbb Q}\cap [0,\infty )}\{\sigma
\leq t\}\cap \{\tau (t)<s\},\qquad s>0.\]
Fix $F\in {\cal B}_{\tau (\sigma )}$ with $P(F)>0$, and define two probability measures $
P_1$ and $P_2$ on 
$D_E[0,\infty )\times C_{[0,\infty )}[0,\infty )\times {\cal L}_U$ by 
\[P_1(C)\equiv\frac 1{P(F)}E^P[{\bf 1}_F{\bf 1}_C],\quad P_2(C)\equiv\frac 
1{P(F)}E^P[{\bf 1}_FE[{\bf 1}_C|Y(\tau (\sigma ))]].\]
Note that 
\[{\cal L}^{P_1}(X^{\sigma}(0))={\cal L}^{P_1}(Y(\tau (\sigma ))=
{\cal L}^{P_2}(Y(\tau (\sigma ))={\cal L}^{P_2}(X^{\sigma}(0))\equiv
\mu .\]
Since 
\[X^{\sigma}(t)=Y^{\tau (\sigma )}(\tau^{\sigma}(t)),\]
where $\tau^{\sigma}$ is given by  
\[\tau^{\sigma}(t)\equiv\inf\{s:\lambda_0^{\tau (\sigma )}(s)>t\}
,\]
and $(Y^{\tau (\sigma )},\lambda_0^{\tau (\sigma )},\Lambda_1^{\tau 
(\sigma )})$ is defined as in $(\ref{shift})$, Lemma 
\ref{taushiftn} yields that ${\cal L}^{P_1}(X^{\sigma})\in\Gamma^{
\infty}_{\mu}$. On the 
other hand ${\cal L}^{P_2}(Y^{\tau (\sigma )},\lambda_0^{\tau (\sigma 
)},\Lambda^{\tau (\sigma )}_1)\in\Pi$ by the optional 
sampling theorem. Moreover, for each $n$,  
\begin{eqnarray*}
&&E^{P_2}[\int_0^{\infty}e^{-\lambda_0^{\tau (\sigma )}(s)}h_n(Y^{
\tau (\sigma )}(s))d\lambda_0^{\tau (\sigma )}(s)]\\
&&=\frac 1{P(F)}E^P[{\bf 1}_FE[\int_0^{\infty}e^{-\lambda_0^{\tau 
(\sigma )}(s)}h_n(Y^{\tau (\sigma )}(s))d\lambda_0^{\tau (\sigma 
)}(s)|Y(\tau (\sigma ))]]\\
&&=\frac 1{P(F)}E^P[E[{\bf 1}_F|Y(\tau (\sigma ))]E[\int_0^{\infty}
e^{-\lambda_0^{\tau (\sigma )}(s)}h_n(Y^{\tau (\sigma )}(s))d\lambda_
0^{\tau (\sigma )}(s)|Y(\tau (\sigma ))]]\\
&&=\frac 1{P(F)}E^P[E[{\bf 1}_F|Y(\tau (\sigma ))]\int_0^{\infty}
e^{-\lambda_0^{\tau (\sigma )}(s)}h_n(Y^{\tau (\sigma )}(s))d\lambda_
0^{\tau (\sigma )}(s)]\\
&&=\gamma (\Pi_{\mu}^{h_1,...,h_{n-1}},h_n),\end{eqnarray*}
where the last equality follows from Lemma 
\ref{taushiftn}. Therefore the distribution of 
$(Y^{\tau (\sigma )},\lambda_0^{\tau (\sigma )},\Lambda_1^{\tau (
\sigma )})$ under $P_2$ belongs to $\Pi^{\infty}$, so that 
${\cal L}^{P_2}(X^{\sigma})\in\Gamma^{\infty}_{\mu}$. Then, by uniqueness of the distribution in 
$\Gamma^{\infty}_{\mu}$, it must hold ${\cal L}^{P_1}(X^{\sigma})
={\cal L}^{P_2}(X^{\sigma})$, which gives 
$(\ref{stM})$. 
\end{proof}\

\begin{remark}
The process constructed in Theorem 
\ref{uniqueness} may not be a solution of the 
constrained (local) martingale problem because 
$(\ref{mgp2})$ is not necessarily a (local) martingale for 
all $f\in {\cal D}$.  However it is, by construction, a solution of 
the martingale problem for $\mathbb A$.  Note that 
${\cal D}(\mathbb A)\supset \{f\vert_{\bar {E}_0}:\,f\in {\cal D}\mbox{\rm \ and }
Bf(x,u)=0,\,\forall (x,u)\in\Xi\cap\partial E_0\}$.  
\end{remark}

\begin{lemma}\label{nonemp}
Let $\nu\in {\cal P}(\bar {E}_0)$.  Suppose  every solution $(Y,\lambda_
0,\Lambda_1)$ 
of the controlled martingale problem for $(A,E_0,B,\Xi )$ with 
initial distribution $\nu$ satisfies $\lambda_0(t)>0$ for all $t>
0$ a.s..  
Then, for every choice of the $\{h_n\}$ in Theorem 
\ref{uniqueness}, $\nu\in {\cal Q}^{\infty}_0$. $ $
\end{lemma}

\begin{proof}
If $(Y,\lambda_0,\Lambda_1)$ has distribution in $\Pi^{\infty}_{\nu}$, then $
\tau (0)=0$.
\end{proof}

\begin{corollary}\label{M-const} 
\item[a)] Let $\nu\in {\cal P}(\bar {E}_0)$. If every solution $(
Y,\lambda_0,\Lambda_1)$ of the controlled 
martingale problem for $(A,E_0,B,\Xi )$ with 
initial distribution $\nu$ satisfies the conditions 
of Lemma \ref{nonemp}\ and
Lemma \ref{invert}, then there exists a strong Markov, 
natural solution to the constrained martingale 
problem for $(A,E_0,B,\Xi )$ with initial distribution $\nu$.  

\item[b)] Let $\nu\in {\cal P}(\bar {E}_0)$. If $\lambda_0$ is a.s. strictly increasing for 
every solution 
of the controlled martingale problem for $(A,E_0,B,\Xi )$ 
with initial distribution $\nu$ (see Lemma 
\ref{increasing} for a sufficient condition), then 
there exists a strong Markov, natural 
solution to the constrained local martingale problem 
for $(A,E_0,B,\Xi )$ with initial distribution $\nu$.  
\end{corollary}

\begin{proof}
By Lemma \ref{nonemp}, $\nu\in {\cal Q}_0^{\infty}$, and the 
assertion follows immediately from Theorem 
\ref{uniqueness} by the same arguments as in Corollary 
\ref{const-exist}.  
\end{proof}

\begin{corollary}\label{Muniq}
Assume Condition \ref{cond1}. 
Let $\nu\in {\cal P}(\bar {E}_0)$. If every solution $(Y,\lambda_
0,\Lambda_1)$ of the controlled 
martingale problem for $(A,E_0,B,\Xi )$ with 
initial distribution $\nu\in {\cal P}(\bar {E}_0)$ satisfies the conditions 
of Lemma \ref{nonemp}\ and
there is a unique (in distribution) 
strong Markov process $X=Y\circ\tau$ with initial distribution $\nu$ 
that can be obtained from a 
solution of the 
controlled martingale problem as in Theorem \ref{const}, 
then there is a unique (in distribution) 
process that can be obtained in this way. 

In particular, under either condition a) or b) of Corollary 
\ref{M-const}, if there is a unique strong Markov, natural 
solution of the constrained (local) martingale problem with 
initial distribution $\nu$, then there exists a unique natural 
solution. 
\end{corollary}

\begin{proof}
If  $\Gamma_{\nu}$ contains more than one distribution, 
then, by selecting appropriate sequences $\{h_n\}$, more than 
one strong Markov solution can be constructed. 
\end{proof}

\begin{remark}\label{natunq}
We can't rule out the possibility that there exist 
solutions of the constrained martingale problem 
that are not natural, but, under Condition 1.2 of 
\cite{KS01}, Theorem 2.2 of that paper yields that for 
any solution of the constrained martingale problem there 
exists a natural solution that has the same one 
dimensional distributions.  By Theorem 3.2 of 
\cite{Kur91}, uniqueness of one dimensional distributions 
for solutions with any given initial distribution implies 
uniqueness of finite dimensional distributions, so under 
Condition 1.2 of \cite{KS01}, uniqueness among natural 
solutions will imply uniqueness among all solutions.  
\end{remark}

\setcounter{equation}{0}
\section{Viscosity solutions}\label{sectvisc}
The approach taken above in the construction of a strong 
Markov solution to the constrained martingale problem 
simplifies the proof of existence of viscosity 
semisolutions to the problem 
\begin{equation}\begin{array}{ll}
v(x)-Av(x)=h(x),&\mbox{\rm for }x\in E_0,\\
Bv(x,u)=0,&\mbox{\rm for }x\in\partial E_0\mbox{\rm \ and some }u
\in\xi_x\end{array}
\label{visc}\end{equation}
given in \cite{CK15}, Section 5.  In fact Theorem \ref{subsol} below 
shows that 
the function $v_h$ defined by $(\ref{gam})$ and Lemma \ref{uexist} is 
a viscosity subsolution of $(\ref{visc})$, and hence the function 
$-v_{-h}$ is a viscosity supersolution. As a consequence, 
under mild assumptions, uniqueness of the strong Markov 
solution of the constrained martingale 
problem starting at each $x\in\bar {E}_0$ implies existence of a viscosity 
solution (Corollary \ref{sol}). This construction
is a ``probabilistic'' alternative  
to Perron's method, and it does not require proving the 
comparison principle for $(\ref{visc})$. 

For unconstrained martingale problems, the analogous 
result follows immediately from Section 3 of \cite{CK15}. 
For a class of jump-diffusion processes, for which 
uniqueness in law holds, \cite{CPD12} proves existence of a 
viscosity solution to the backward Kolmogorov equation directly, 
and then uniqueness of the viscosity solution by the comparison principle. 
The fact that the comparison principle for $(\ref{visc})$ implies 
uniqueness of the 
solution to the constrained (or unconstrained) martingale 
problem is the object of \cite{CK15}.

\begin{theorem}\label{subsol}
Let $(Y,\lambda_0,\Lambda_1)$ be a solution to the controlled martingale 
problem for \break 
$(A,E_0,B,\Xi )$. For $h\in C(\bar {E}_0)$, let $v\equiv 
v_h$ be 
the function defined by $(\ref{gam})$ and 
Lemma \ref{uexist}. 

Then $v\big\vert_{\bar {E}_0}$ is a viscosity subsolution of 
$(\ref{visc})$, that is, it is upper semicontinuous, and
if $f\in {\cal D}$ and $x\in\bar {E}_0$ satisfy 
\begin{equation}\sup_{z\in\bar {E}_0}(v-f)(z)=(v-f)(x),\label{csubtest}\end{equation}
then 
\[\begin{array}{ll}
v(x)-Af(x)\leq h(x),&\mbox{\rm \ if }x\in E_0\cup (\partial E_0-F_
1),\\
(v(x)-Af(x)-h(x))\wedge (-\max_{u\in\xi_x}Bf(x,u))\leq 0,&\mbox{\rm \ if }
x\in\partial E_0\cap F_1,\end{array}
\]
($\xi_x$ and $F_1$ being defined at the beginning of Section 
\ref{ClMP}).
\end{theorem}

\begin{proof}
$ $$v$ is upper semicontinuous by Lemma \ref{uexist}. 

Suppose $x$ is a point such that 
$v(x)-f(x)=\sup_z(v(z)-f(z))$.  As we can always 
add a constant to $f$, we can assume $v(x)-f(x)=0$.  
By compactness, we have 
\[v(x)=E^P\left[\int_0^{\infty}e^{-\lambda_0(s)}h(Y(s))d\lambda_0
(s)\right]\]
for some $P\in\Pi_{\delta_x}$.
For $\epsilon >0$, define 
\[\tau_{\epsilon}=\epsilon\wedge\inf\{t>0:r(Y(t),x)\geq\epsilon\mbox{\rm \ or }
r(Y(t-),x)\geq\epsilon \},\]
where $r$ is the metric in $E$, and let $H_{\epsilon}=e^{-\lambda_
0(\tau_{\epsilon})}$. 
Since $(Y,\lambda_0,\Lambda_1)$ is a solution to the controlled martingale 
problem for $(A,E_0,B,\Xi )$, we have 
\begin{eqnarray*}
0&&=v(x)-f(x)\\
&&=E^P\left[\int_0^{\infty}e^{-\lambda_0(s)}(h-f+Af)(Y(s))\,d\lambda_
0(s)+\int_{[0,\infty )\times U}e^{-\lambda_0(s)}Bf(Y(s),u)\Lambda_
1(ds\times du)\right]\\
&&=E^P\left[\int_0^{\tau_{\epsilon}}e^{-\lambda_0(s)}(h-f+Af)(Y(s
))\,d\lambda_0(s)+\int_{[0,\tau_{\epsilon}]\times U}^{}e^{-\lambda_
0(s)}Bf(Y(s),u)\Lambda_1(ds\times du)\right]\\
&&\qquad +E^P\left[e^{-\lambda_0(\tau_{\epsilon})}\int_0^{\infty}
e^{-\lambda_0^{\tau_{\epsilon}}(s)}(h-f+Af)(Y^{\tau_{\epsilon}}(s
))\,d\lambda_0^{\tau_{\epsilon}}(s)\right]\\
&&\qquad +E^P\left[e^{-\lambda_0(\tau_{\epsilon})}\int_{[0,\infty 
)\times U}e^{-\lambda_0^{\tau_{\epsilon}}(s)}Bf(Y^{\tau_{\epsilon}}
(s),u)\Lambda_1^{\tau_{\epsilon}}(ds\times du)\right]\\
&&=E^P\left[\int_0^{\tau_{\epsilon}}e^{-\lambda_0(s)}(h-f+Af)(Y(s
))\,d\lambda_0(s)\right]\\
&&\qquad +E^P\left[\int_{[0,\tau_{\epsilon}]\times U}e^{-\lambda_
0(s)}Bf(Y(s),u)\Lambda_1(ds\times du)\right]\\
&&\\
&&\qquad +E^P[H_{\epsilon}]E^{P^{\tau_{\epsilon},H_{\epsilon}}}\left
[\int_0^{\infty}e^{-\lambda_0(s)}(h-f+Af)(Y(s))\,d\lambda_0(s)\right
]\\
&&\qquad +E^P[H_{\epsilon}]E^{P^{\tau_{\epsilon},H_{\epsilon}}}\left
[\int_{[0,\infty )\times U}e^{-\lambda_0(s)}Bf(Y(s),u)\Lambda_1(d
s\times du)\right],\\
\end{eqnarray*}
with $(Y^{\tau_{\epsilon}},\lambda_0^{\tau_{\epsilon}}(s),\Lambda_
1^{\tau_{\epsilon}})$ and $P^{\tau_{\epsilon},H_{\epsilon}}$ as in Lemma \ref{ctrans}. 
Setting $\mu_{\epsilon}(\cdot )\equiv P^{\tau_{\epsilon},H_{\epsilon}}
(Y(0)\in\cdot )=P(Y(\tau_{\epsilon})\in\cdot )$, and denoting 
$\mu_{\epsilon}f\equiv\int_{F_2}f(z)\,\mu_{\epsilon}(dz)$, by Lemma \ref{ctrans} and Lemma 
\ref{uexist} we have 
\begin{eqnarray*}
&&E^P[H_{\epsilon}]E^{P^{\tau_{\epsilon},H_{\epsilon}}}\left[\int_
0^{\infty}e^{-\lambda_0(s)}\left(h(Y(s))-f(Y(s))+Af(Y(s))\right)d
\lambda_0(s)\right]\\
&&\quad +E^P[H_{\epsilon}]E^{P^{\tau_{\epsilon},H_{\epsilon}}}\left
[\int_{[0,\infty )\times U}e^{-\lambda_0(s)}Bf(Y(s),u)\Lambda_1(d
s\times du)\right]\\
&&\leq E^P[H_{\epsilon}](\gamma (\Pi_{\mu_{\epsilon}},h)-\mu_{\epsilon}
f)=E^P[H_{\epsilon}](\mu_{\epsilon}v-\mu_{\epsilon}f)\leq 0,\end{eqnarray*}
where the last inequality uses the fact that 
$v-f\leq 0$. Therefore 
\begin{eqnarray*}
0&\leq&\lim_{\epsilon\rightarrow 0}\frac {E^P\left[\int_0^{\tau_{
\epsilon}}e^{-\lambda_0(s)}(h-f+Af)(Y(s))\,d\lambda_0(s)+\int_{[0
,\tau_{\epsilon}]\times U}e^{-\lambda_0(s)}Bf(Y(s),u)\Lambda_1(ds
\times du)\right]}{E^P[\tau_{\epsilon}]}\\
&=&h(x)-f(x)+Af(x_0)=h(x)-v(x)+Af(x),\end{eqnarray*}
if $x\in E_0\cup (\partial E_0-F_1)$, and 
\[0\leq (h(x)-v(x)+Af(x))\vee\max_{u\in\xi_x}Bf(x,u)),\]
if $x\in\partial E_0\cap F_1$.
\end{proof}

\begin{remark}\label{Mrepres}
Note that, for each $x\in\bar {E}_0$, 
\[v(x)\equiv v_h(x)=E\left[\int_0^{\infty}e^{-s}h(X^h(s))ds\right
]\]
for some strong Markov process $X^h=Y\circ\tau$  
obtained from a solution of the controlled martingale 
problem as in Theorem \ref{const} with $Y(0)=x$.
\end{remark}

\begin{corollary}\label{sol}

\item[a)] If, for each $x\in\bar {E}_0$, there is a unique solution $
(Y,\lambda_0,\Lambda_1)$ 
of the controlled martingale problem for $(A,E_0,B,\Xi )$ with 
$Y(0)=x$, then there exists a viscosity solution to $(\ref{visc})$.

\item[b)] If the assumptions 
of Corollary \ref{M-const} a) or b) are satisfied for each $\delta_x$, 
$x\in\bar {E}_0$, and there is a unique strong Markov, natural 
solution to the (local) constrained 
martingale problem with $X(0)=x$, 
then there exists a viscosity solution to $(\ref{visc})$.
\end{corollary}

\begin{proof}
For each $x\in\bar {E}_0$, let $v\equiv v_h$ be 
the function defined by $(\ref{gam})$ and 
Lemma \ref{uexist}. Then, by uniqueness of the solution 
to the controlled martingale problem for $(A,E_0,B,\Xi )$, 
\[v(x)\equiv v_h(x)=-v_{-h}(x),\]
and, as noted at the beginning of this subsection, $-v_{-h}$ is a 
supersolution of $(\ref{visc})$. 

The second assertion follows from Remark \ref{Mrepres} by the 
same argument.

\end{proof}

\setcounter{equation}{0} 

\section{Diffusions with oblique reflection in piecewise smooth 
domains: existence and Markov property}\label{reflect}

Let $E_0$ be a bounded, 
simply connected, open subset of ${\mathbb R}^d$ such that 
$E_0\equiv\cap_{i=1}^mE_0^i$, where $E_0^i$, $i=1,...,m$, are 
simply connected open sets in $\R^d$ with $C^1$ boundaries.  
Specifically, we will assume that for each $i$ there is a 
function $\psi_i\in C^1({\mathbb R}^d)$ such that $E_0^i=\{x:\psi_i(
x)>0\}$ and that 
$\psi_i(x)=0$ implies $\nabla\psi_i(x)\neq 0$. In particular, 
$\partial E_0^i=\{x:\psi_i(x)=0\}$, and the inward normal at $x\in
\partial E_0^i$ is 
$n^i(x)=\frac {\nabla\psi_i(x)}{|\nabla\psi_i(x)|}$. We will assume that 
\begin{equation}\overline {E_0}=\cap_{i=1}^m\overline {E_0^i}.\label{closure}\end{equation}

Suppose that on  $\partial E_0^i$ a variable direction of reflection $
g^i$ is 
assigned.  We assume that $g^i$ is continuous on $\partial E^i_0$ and 
$\langle \nabla\psi^i(x),g^i(x)\rangle >0$, $x\in\partial E_0^i$.  It is convenient, and no loss of 
generality to assume that $g^i:{\mathbb R}^d\rightarrow {\mathbb R}^d$ and is continuous 
on all of ${\mathbb R}^d$ with $\langle \nabla\psi^i(x),g^i(x)\rangle\geq 
0$ (allowing $0$ away 
from $\partial E_0^i$).
Noting that $x\in\partial E_0$ may be in more than 
one $\partial E^i_0$, for $x\in\partial E_0$, we define the cone of 
possible directions of reflection
\begin{equation}G(x)\equiv\left\{\sum_{i:x\in\partial E_0^i}\eta_
ig^i(x),\,\eta_i\geq 0\right\}\label{rcone}\end{equation}
and also define
\begin{equation}N(x)\equiv\left\{\sum_{i:x\in\partial E_0^i}\eta_
in^i(x),\eta_i\geq 0\right\}.\label{Ndef}\end{equation}

Starting from the late '70s, there has been a considerable 
amount of work devoted to proving existence and 
uniqueness of reflecting diffusions in $\bar {E}_0$ with direction 
of reflection $g^i$ on $\partial E_0^i$.  Perhaps the most general result 
in this sense is \cite{DI93}.  However the assumptions in 
\cite{DI93} are not satisfied in many natural situations, 
as in the following example.  

\begin{example}\label{ex}
Let $E_0\equiv E_0^1\cap E_0^2$, where $E_0^1$ is the unit ball centered at $
(1,0)$ 
and $E_0^2$ is the upper half plane. Let $n^i$, $i=1,2$, denote the 
unit, inward normal to $\overline {E_0^i}$, and 
\[g^i(x)\equiv\left[\begin{array}{cc}
\cos(\vartheta )&\sin(\vartheta )\\
-\sin(\vartheta )&\cos(\vartheta )\end{array}
\right]n^i(x),\qquad\vartheta\mbox{\rm \ a constant angle, }\frac {
\pi}4\leq\vartheta <\frac {\pi}2.\]
Then, at $x^0=0$, it can be proved by contradiction that 
there is no convex compact set that satisfies $(3.7)$ of 
\cite{DI93}. 
\end{example}

In addition \cite{DI93} 
does not cover the case of cusp like singularities  
(covered by \cite{CK18} in dimension 2). 

\cite{DW95} considers convex polyhedrons (take 
$\psi_i(x)=\langle n^i,x\rangle -b^i$, $n^i$ and $b^i$ constant) with constant 
direction of reflection $g^i$ on each face. In this context, 
\cite{DW95}
proves existence and uniqueness (in distribution) 
of semimartingale reflecting 
Brownian motion under a condition 
which, in the case of simple polyhedrons, reduces to the 
assumption that, 
for every $x\in\partial E_0$, there exists $e(x)\in N(x)$, $|e(x)
|=1$, such that 
\begin{equation}\langle g,e(x)\rangle >0,\quad\forall g\in G(x)-\{
0\}.\label{S1}\end{equation}
Moreover, for simple polyhedrons, \cite{DW95}, 
Propositions 1.1 and  
1.2, shows that $(\ref{S1})$ is necessary for existence of 
semimartingale reflecting Brownian motion. 
(Non-semimartingale reflecting Brownian motion, which is studied, 
for example, in \cite{KR10}, \cite{KR17} and \cite{LRZ19}, is not 
considered here.) Note that $(\ref{S1})$ is 
satisfied in Example \ref{ex}. 

In \cite{DW95}, 
a key point in proving uniqueness is the fact that 
there exist strong Markov processes that satisfy the definition 
of semimartingale reflecting Brownian motion  
and that uniqueness among these strong 
Markov processes implies uniqueness among all processes 
that satisfy the definition 
(analogously in \cite{KW91} and \cite{TW93}).  
Our goal here is to prove that this key point holds 
for general diffusion processes on
domains $E_0$ as defined above 
under Condition \ref{dircond} below, thus 
providing the first step in extending proofs of 
uniqueness to this more general setting

In \cite{DI93}, \cite{DW95} and in most of the literature, 
reflecting diffusions are defined as (weak) solutions of 
stochastic differential equations with reflection. 
Here we start by studying the corresponding controlled 
martingale problem 
and constrained martingale problem, and then 
show that the set 
of natural solutions to the constrained martingale problem 
coincides with the set of solutions of the 
stochastic differential equation with reflection.

We consider the controlled martingale problem 
for $(A,E_0,B,\Xi )$, with 
\begin{eqnarray}
&Af(x)\equiv\langle\nabla f(x),b(x)\rangle +\frac 12\mbox{\rm tr}\big
(\sigma (x)\sigma^T(x)D^2f(x)\big),&\non\\
&Bf(x,u)\equiv\langle\nabla f(x),u\rangle ,&\label{rfclmp}\\
&U\equiv \{u\in\R^d:\,|u|=1\},&\non\\
&\Xi\equiv \{(x,u)\in\partial E_0\times U:\,u\in G(x)\},&\non\end{eqnarray}
${\cal D}\equiv C^2_c(\mathbb R^d)$, and we assume that $\sigma$ and $
b$ are 
bounded and continuous on ${\mathbb R}^d$.

Note that $F_1$, defined at the beginning of Section 
\ref{ClMP}, in this case is $\partial E_0$, so a solution of the controlled 
martingale problem must take values in $\bar {E}_0$ (Remark 
\ref{stsp}).

For $x\in (E_0)^c=\cup_{i=1}^m(E_0^i)^c$, let 
\begin{equation}I(x)\equiv \{i:x\in (E_0^i)^c\}.\label{indexset}\end{equation}
Since $(E_0^j)^c$ is closed, 
if $j\in I(z^k)$ for some sequence $z^k\rightarrow x$, then 
$j\in I(x)$. Consequently, for each $x\in (E_0)^c$ there exists $
\delta (x)$ such that 
\begin{equation}I(z)\subset I(x),\quad\mbox{\rm for }z\in (E_0)^c\mbox{\rm \ with }
|z-x|<\delta (x).\label{Iusc}\end{equation}
Note that, for $x\in\partial E_0$, 
\begin{equation}I(x)\equiv \{i:x\in\partial E_0^i\}.\label{bdindset}\end{equation}
Define also, for $x\in\partial E_0$, 
\begin{equation}{\cal I}(x)\equiv \{I\subset I(x):\,\exists z\in 
(\overline {E_0})^c,\,|z-x|<\delta (x),\,s.t.\,I=I(z)\}.\label{scriptI}\end{equation}

We assume that $E_0^i$ and $g^i$, $i=1,...,m$, satisfy the 
following condition. 

\begin{condition}\label{dircond}
\item[a)]For $i=1,...,m$, $g^i:\R^d\rightarrow\R^d$ are 
continuous vector fields of unit length on $\partial E_0^i$, that satisfy  
\[\langle g^i(x),n^i(x)\rangle >0,\quad\forall x\in\partial E_0^i
.\]
\item[b)]For each $x\in\partial E_0$, there exists $e(x)\in N(x)$, $
|e(x)|=1$, that 
satisfies 
\[\langle g,e(x)\rangle >0,\quad\forall g\in G(x)-\{
0\}.\]

\item[c)]For each $x\in\partial E_0$,  $I\in {\cal I}(x)$, and 
$n=\sum_{i\in I}\eta_in^i(x)$, $\eta_i\geq 0$, $\sum_{i\in I}\eta_
i>0$, there exists $j\in I$ such that 
\[\langle n,g^j(x)\rangle >0.\]
\end{condition}
\vskip.1in

\begin{remark}\label{polyhedron} 
In the case of simple, convex polyhedrons with constant 
direction of reflection on each face, 
Condition \ref{dircond} b) becomes {\bf (S.b)} of 
\cite{DW95} and Condition \ref{dircond} c) is immediately 
implied by {\bf (S.a)} of \cite{DW95}. 
In fact, since {\bf (S.a)} and {\bf (S.b)} are equivalent for simple 
polyhedrons 
(\cite{DW95}, Proposition 1.1), in this case Condition \ref{dircond} 
is equivalent to the assumptions of \cite{DW95}. 
\end{remark}

\begin{example}\label{cuspeg}
For domains with curved boundaries and singularities, e.g. 
cusp-like singularities, Condition \ref{dircond} may be 
satisfied, whereas {\bf (S.a)} and {\bf (S.b)} of \cite{DW95} are not. 
As an example, consider the domain 
\[E_0\equiv \{x\in\R^2:\,0<x_1,\,-x_1^4<x_2<x_1^2,\,x_1^2+x_2^2<1
\}.\]
Then $E_0=\cap_{i=1}^4E^i_0$ with 
\[\psi_1(x)\equiv x_2+x_1^4,\,\,\,\psi_2(x)\equiv x_1^2-x_2,\,\,\,
\psi_3(x)\equiv 1-x_1^2-x_2^2,\,\,\,\psi_4(x)\equiv x_1.\]
Let $g^1$ and $g^2$ be 
continuous vector fields defined on $\partial E_0^1$ and $\partial 
E_0^2$, 
respectively, such that $g^1(0)=[-\frac 12,\frac {\sqrt {3}}2]^T$, 
$g^2(0)=[\frac {\sqrt {2}}2,-\frac {\sqrt {2}}2]^T$, and take $g^
4(0)\equiv [1,0]^T$. Then 
${\cal I}(0)=\big\{\{1\},\{2\},\{4\},\{1,4\},\{2,4\}\big\}$ and it is easy to check 
that Condition \ref{dircond} is satisfied at $0$. 
\end{example}

\begin{remark}
In general, there are multiple possible choices of $E_0^i$ and 
$g^i$, $i=1,...,m$, that determine the same domain $E_0=\cap_{i=1}^
mE_0^i$ 
and the 
same direction of reflection at each point of the 
smooth part of the boundary of $E_0$. In some cases, 
some of these choices satisfy Condition \ref{dircond} and 
others do not. For instance, in Example \ref{cuspeg} one 
can take $E_0=\cap_{i=1}^3\tilde {E}^i_0$ with $\tilde{\psi}_3=\psi_
3$ and 
\[\tilde{\psi}_1(x)\equiv\left\{\begin{array}{ll}
x_2+x_1^4,&x_1\geq 0,\\
x_2-x_1^4,&x_1<0,\end{array}
\right.,\quad\tilde{\psi}_2(x)\equiv\left\{\begin{array}{ll}
x_1^2-x_2,&x_1\geq 0,\\
-x_1^2-x_2,&x_1<0.\end{array}
\right.\]
Then ${\cal I}(0)=\big\{\{1\},\{2\},\{1,2\}\big\}$ and, with the same $
g^1(0)$ and 
$g^2(0)$ as above, Condition \ref{dircond} is not satisfied at 
$0$. 
\end{remark}

As anticipated in Remark \ref{exist}, we will obtain a solution 
to the controlled martingale 
problem $(\ref{rfclmp})$ by constructing a solution to the 
corresponding {\em patchwork martingale problem\/} (\cite{Kur90}), which will 
also be a solution to the controlled martingale problem. 

\begin{definition}\label{pwmp} {\bf (\cite{Kur90}, Lemma 1.1) }
Given a complete, separable metric space $E$, an open 
subset $E_0$ of $E$, a partition of $E-E_0$ into Borel sets 
$\{E_1,...,E_m\}$ and dissipative operators 
$A,B_1,...,B_m\subset C(E)\times C(E)$, each containing $(1,0)$ and 
with a common domain ${\cal D}$ dense in $C(E)$, a solution to the 
patchwork martingale problem for $(A,E_0,B_1,E_1,...B_m,E_m)$ is 
a process $(Y,\lambda_0,l_1,...,l_m)$ such that $Y$ has paths in 
${\cal D}_E[0,\infty )$, $\lambda_0,l_1,...,l_m$ are nondecreasing, $
l_1$ increases only 
when $Y\in\bar {E}_i$, $\lambda_0(t)+\sum_{i=1}^ml_i(t)=t$, and there exists a 
filtration $\{{\cal F}_t\}$ such that $(Y,\lambda_0,l_1,...,l_m)$ is $
\{{\cal F}_t\}$-adapted 
and 
\[f(Y(t))-f(Y(0))-\int_0^tAf(Y(s))d\lambda_0(s)-\sum_{
i=1}^m\int_0^tB_if(Y(s))dl_i(s)\]
is a $\{{\cal F}_t\}$-martingale for all $f\in {\cal D}$. 
\end{definition}

\begin{theorem}\label{ctrl-exist} 
For each $\nu\in {\cal P}(\bar {E}_0)$, there exists a solution $
(Y,\lambda_0,\Lambda_1)$ of 
the controlled martingale problem for $(A,E_0,B,\Xi )$ defined 
by $(\ref{rfclmp})$, 
with initial distribution $\nu$.  
\end{theorem}

\begin{proof}
Let $\chi :\R\rightarrow\R$ be a $C^{\infty}$ function 
such that $\chi (r)=0$ for $r\leq 0$, $\chi (r)=1$ for $r\geq 1$, $
\chi'(r)>0$ 
for $0<r<1$, and define 
\begin{equation}\phi (x)\equiv\sum_{i=1}^m\chi (-\psi_i(x)).\label{phi}\end{equation}
For $x\in\partial E_0$ and $I\in {\cal I}(x)$,  let
\[N^I_1(x)=\{n:\,n=\sum_{i\in I}\eta_in^i(x),\,\eta_i\geq 0,\,\sum_{
i\in I}\eta_i=1\},\]
and define
\[\beta\equiv\inf_{x\in\partial E_0}\min_{I\in {\cal I}
(x)}\inf_{n\in N_1^I(x)}\max_{j\in I}\,\,\langle n,g^j(x)\rangle 
.\]
By Condition \ref{dircond}c) and compactness,  
\begin{equation}\beta >0.\label{beta}\end{equation}

Let $0<\epsilon_0\leq 1$ be sufficiently small so that for all $i
=1,...,m$, 
\[\inf_{x\in (\overline {E_0})^c:\,d(x,\overline {E_0})\leq\epsilon_
0}\psi_i(x)>-1,\inf_{x\in (\overline {E_0})^c:\,d(x,\overline {E_
0})\leq\epsilon_0}|\nabla\psi_i(x)|>0,\]
and  for $|x-z|\leq\epsilon_0$, $d(x,\partial E_0^i\cap\partial E_
0)\leq\epsilon_0,$ $d(z,\partial E_0^i\cap\partial E_0)\leq\epsilon_
0$, 
 
\[|g^i(x)-g^i(z)|\leq\frac {\beta}4\quad\mbox{\rm and}\quad |n^i(
x)-n^i(z)|\leq\frac {\beta}4.\]
Then, in particular, by $(\ref{closure})$, for $x\in (\overline {
E_0})^c$, $d(x,\overline {E_0})\leq\epsilon_0$,  
\begin{equation}\phi (x)>0,\quad\sum_{i\in I(x)}\chi'(-\psi_i(x))
|\nabla\psi_i(x)|>0.\label{phi-pos}\end{equation}

For $z\in\partial E_0$, let $\delta (z)$ be as in $(\ref{Iusc})$. 
By compactness, there exists $\delta_0>0$ such that, 
for every $x\in (\overline {E_0})^c$ with  $d(x,\overline {E_0})\leq
\delta_0$, there exists 
$z\in\partial E_0$ such that $|x-z|<\delta (z)$, hence $I(x)\in {\cal I}
(z)$. 

For each $j=1,...,m$, $x\in (\overline {E_0})^c$ with $d(x,\bar {
E}_0)\leq\delta_0\wedge\epsilon_0$, and 
$z\in\partial E_0$ with $|x-z|<\delta (z)$,
\begin{eqnarray*}
-\langle \nabla\phi (x),g^j(x)\rangle&=&\sum_{i\in I(x)}\chi'(-\psi_i(
x))|\nabla\psi_i(x)|\,\frac {\langle \nabla\psi_i(x),g^j(x)\rangle}{|\nabla\psi_
i(x)|}\\
&=&\sum_{i\in I(x)}\chi'(-\psi_i(x))|\nabla\psi_i(x)|\,\langle n^i(x),
g^j(x)\rangle\\
&\geq&\sum_{i\in I(x)}\chi'(-\psi_i(x))|\nabla\psi_i(x)|\left(\langle 
n^i(z),g^j(z)\rangle -|\langle n^i(x),g^j(x)\rangle -\langle n^i(
z),g^j(z)\rangle |\right)\\
&\geq&\sum_{i\in I(x)}\chi'(-\psi_i(x))|\nabla\psi_i(x)|\left(\langle n,g^
j(z)\rangle -\frac {\beta}2\right),\end{eqnarray*}
where 
\[n\equiv\frac 1{\sum_{i\in\stackrel {}I(x)}\chi'(-\psi_i(x))|\nabla\psi_
i(x)|}\sum_{i\in\stackrel {}I(x)}\chi'(-\psi_i(x))|\nabla\psi_i(x)|n^i
(z)\]
belongs to $N_1^{I(x)}(z)$ since $I(x)\in {\cal I}(z)$. 
$(\ref{beta})$ implies that for some $j\in I(x)$,  
\begin{equation}-\langle \nabla\phi (x),g^j(x)\rangle\geq\sum_{i\in I(
x)}\chi'(-\psi_i(x))|\nabla\psi_i(x)|\,\frac {\beta}2>0.\label{pos}\end{equation}

Define
\[E\equiv \{x:\,d(x,\bar {E}_0)\leq\delta_0\wedge\epsilon_0\}\]
and
\begin{equation}\tilde {F}_i=\{x\in E:\psi_i(x)\leq 0,\,\langle\nabla
\phi (x),\,g^i(x)\rangle\leq 0\}.\label{pwk-predom}\end{equation}
By $(\mbox{\rm \ref{pos}})$, each $x\in E-E_0$ 
is in at least one of the $\tilde {F}_i$, so defining 
\[\tilde {E}_1=\{x\in E:\psi_1(x)\leq 0,\,\langle\nabla\phi (x),\,
g^1(x)\rangle\leq 0\}\]
and
\[\tilde {E}_i=\{x\in E:\psi_i(x)\leq 0,\,\langle\nabla\phi (x),\,
g^i(x)\rangle\leq 0\}-\cup_{j<i}\tilde {E}_j,\quad i=2,\ldots ,m,\]
$E_0,\tilde {E}_1,\ldots ,\tilde {E}_m$ are disjoint and
\[E=E_0\cup\bigcup_{i=1}^m\tilde {E}_i.\]

Setting $\tilde {{\cal D}}=C^2(E)$, $\rho (x)=[1-\chi (\frac {d(x
,\bar {E}_0)}{\delta_0\wedge\epsilon_0})]$, $\tilde {A}f(x)=\rho 
(x)Af(x)$, and 
$\tilde {B}_if=\rho (x)\langle\nabla f(x),g^i(x)\rangle$, $\tilde {
A}$ and the $\tilde {B}_i$ are 
dissipative, and Lemma 1.1 of \cite{Kur90} yields that, 
for each $\nu\in {\cal P}(\bar {E}_0)$, there exists a solution, $
(Y,\lambda_0,l_1,...,l_m)$, 
of the patchwork martingale problem for 
$(\tilde {A},E_0,\tilde {B}_1,\tilde {E}_1,...\tilde {B}_m,\tilde {
E}_m)$ with initial distribution $\nu$.  
Then, for  $f\in\tilde {{\cal D}}$
\begin{eqnarray*}
M_f(t)&=&f(Y(t))-\int_0^tAf(Y(s))d\lambda_0(s)-\sum_{i=1}^m\int_0^
t\tilde {B}_if(Y(s))dl_i(s)\\
&=&f(Y(t))-\int_0^tAf(Y(s))d\lambda_0(s)-\sum_{i=1}^m\int_0^t\rho 
(Y(s))\langle \nabla f(Y(s),g^i(Y(s)\rangle dl_i(s)\end{eqnarray*}
is a $\{{\cal F}_t\}$-martingale. (We can write $A$ rather than $
\tilde {A}$ 
since $Af=\tilde {A}f$ on $\bar {E}_0$.)

Since $\phi$ is constant on $\bar {E}_0$, if $\phi$ were $C^2$, then
\begin{equation}M_{\phi}(t)=\phi (Y(t))-\sum_{i=1}^m\int_0^t\rho 
(Y(s))\langle \nabla\phi (Y(s),g^i(Y(s)\rangle dl_i(s)\label{testmgp}\end{equation}
would be a martingale.  Since we can approximate $\phi$ by $C^2$ 
functions $\{\phi^n\}$ in such a way that $\phi^n$ is constant on $
\bar {E}_0$ and 
$\nabla\phi^n\rightarrow \nabla\phi$ uniformly on $E$,  $M_{\phi}$ is a martingale even if  $
\phi$ 
is not $C^2$.  $M_{\phi}$ is a nonnegative 
martingale because $\langle\nabla\phi ,g^i\rangle\leq 0$ on $\bar{
\tilde {E}}_i$. If $Y(0)\in\bar {E}_0$, then $M_{\phi}(0)=0$ 
so, as in the proof of Lemma 1.4 of \cite{Kur90}, 
$M_{\phi}(t)=0$ for all $t\geq 0$.  As all terms in $M_{\phi}$ are 
nonnegative, $\phi (Y(t))$ must be zero for all $t\geq 0$, and hence, 
by $(\ref{phi-pos})$, $Y(t)\in\bar {E}_0$ for all $t\geq 0$.  Therefore 
$(Y,\lambda_0,l_1,...,l_m)$ is a solution of the patchwork martingale 
problem for $(A,E_0,B_1,E_1,...B_m,E_m)$, where 
\begin{eqnarray}
&E_1\equiv \{x\in\partial E_0:\,\psi_1(x)=0\}&\non\\
&E_i\equiv \{x\in\partial E_0:\,\psi_1(x)>0,\,..,\,\psi_{i-1}(x)>
0,\,\psi_i(x)=0\},\quad i=2,...,m,&\label{rfpwmp}\\
&B_if(x)\equiv\langle\nabla f(x),g^i(x)\rangle &\non.
\end{eqnarray}
If we define 
\begin{equation}\Lambda_1([0,t]\times C)\equiv\sum_{i=1}^m\int_0^
t{\bf 1}_C(g^i(Y(s))dl_i(s)\label{pwk-to-ctrl}\end{equation}
then $(Y,\lambda_0,\Lambda_1)$ is a solution of the controlled 
martingale problem for $(A,E_0,B,\Xi )$.
\end{proof}

Let $(Y,\lambda_0,\Lambda_1)$ be a solution of the controlled martingale 
problem for $(A,E_0,B,\Xi )$. It is easy to verify that $Y$ 
is continuous and 
\begin{equation}Y(t)-Y(0)-\int_0^tb(Y(s))d\lambda_0(s)-\int_{[0,t
]\times U}u\,\Lambda_1(ds\times du)\equiv M(t)\label{Y}\end{equation}
is a continuous martingale with 
$[M](t)=\int_0^t(\sigma\sigma^T)(Y(s))d\lambda_0(s)$.

The following lemma is the analog of Lemma 3.1 of 
\cite{DW95} and its proof is based on similar arguments. 

\begin{lemma}\label{takeoff}
For every solution $(Y,\lambda_0,\Lambda_1)$ of the controlled martingale 
problem for $(A,E_0,B,\Xi )$ defined by $(\ref{rfclmp})$,
 $\lambda_0(t)>0$ for all $t>0$, a.s.. 
\end{lemma}

\begin{proof}\ 
By (\ref{Y}), for $\tau (0)=\inf\{t\geq 0:\,\lambda_0(t)>0\}$,
\begin{equation}Y(t\wedge\tau (0))=Y(0)+\int_{[0,t\wedge\tau (0)]\times U}u\,\Lambda_
1(ds\times du).\label{YLip}\end{equation}
For every path such that $\tau (0)>0$, 
$\lambda_1(t\wedge\tau (0))=t\wedge\tau (0)$, and we must have $Y
(t)\in\partial E_0$ for all 
$t\in [0,\tau (0))$. Setting, for $I\subset \{1,...,m\}$, 
\[\partial_IE_0\equiv \{x\in\partial E_0:\,I(x)=I\}\]
there must exist a $k$ such that 
\begin{equation}Y(t)\in\bigcup_{I:\,|I|\geq k}\partial_IE_0\label{yin}\end{equation}
for all 
$t\in [0,\tau (0))$. Let $k_0$ be the maximal such $k$, that is, $
k=k_0$ 
satisfies (\ref{yin}) and there exists $t\in [0,\tau (0))$ such that
 $|I(Y(t))|=k_0$. 

By $(\ref{Iusc})$ and the continuity of $Y$, for $s>t$ close 
enough to $t$, $I(Y(r))\subset I(Y(t))$ for all $r\in [t,s]$.  Since by 
definition of $k_0$, $|I(Y(r))|\geq k_0$, we must have 
$I(Y(r))=I(Y(t))$ for all $r\in [t,s]$.  

Since $\partial E_0^i$ is $C^1$, 
\[|\langle Y(s)-Y(t),n^i(Y(t))\rangle |=o(|Y(s)-Y(t)|),\quad\forall 
i\in I(Y(t)).\]
In addition, by (\ref{YLip}) and the fact that 
$\lambda_1(t)=\Lambda_1([0,t]\times U)$ is Lipschitz,
\[|Y(s)-Y(t)|\leq O(s-t),\]
so that
\[\big|\int_{(s,t]\times U}\langle u,n\rangle\,\Lambda_1(ds\times 
du)\big|=|\langle Y(s)-Y(t),n\rangle |=o(|s-t|),\quad\forall n\in 
N(Y(t)).\]
On the other hand, setting  
$G^{I(y)}(x)\equiv\big\{\sum_{i\in I(y)}\eta_ig^i(x),\,\eta_i\geq 
0\big\}$, $(\ref{Iusc})$ implies 
$G(Y(r))\subset G^{I(Y(t))}(Y(r))$ for all $r\in [t,s]$. Since the 
Hausdorff distance 
$d(G^{I(y)}(x)\cap U,G(y)\cap U)\rightarrow 0$ as $x\rightarrow y$, if $
s$ is close enough to $t$, by 
\ref{dircond} b) we have, for some $e\in N(Y(t))$, $|e|=1$, 
\[\langle Y(s)-Y(t),e\rangle =\int_{(t,s]\times U}\langle u,e\rangle\,
{\bf 1}_{\Xi}(Y(r),u)\,\Lambda_1(dr\times du)\geq\frac 12\inf_{u\in 
G(Y(t))\cap U}\langle u,e\rangle (s-t).\]

Consequently, by contradiction, $\tau (0)$ must be zero almost 
surely.
\end{proof}

\begin{lemma}\label{standass}
The controlled martingale problem for $(A,E_0,B,\Xi )$ defined by $
(\ref{rfclmp})$ 
satisfies Condition \ref{cond1}. 
\end{lemma}

\begin{proof}
Condition \ref{cond1} a) is clearly satisfied. 
Condition \ref{cond1} b) is satisfied by Theorem  
\ref{ctrl-exist}, while Condition \ref{cond1} c) is satisfied 
by Lemma \ref{takeoff} and Lemma \ref{laminf}. 
\end{proof}

\begin{theorem}\label{rfcmp-exist}
For each $\nu\in {\cal P}(\bar {E}_0)$ there exists a natural solution of the 
constrained martingale problem for $(A,E_0,B,\Xi )$ defined by $(\ref{rfclmp})$. 
\end{theorem}

\begin{proof}
By Lemma \ref{takeoff} and 
Lemma \ref{increasing}, Corollary \ref{const-exist} b) applies. 
\end{proof}
\vskip.1in

As mentioned at the beginning of this section, a 
reflecting diffusion in $\bar {E}_0$ with direction of reflection $
g^i$ 
on $\{x\in\partial E_0:\,\psi_i(x)=0,\,\psi_j(x)>0,\mbox{\rm \ for }
j\neq i\}$, $i=1,...,m$, 
is often defined as a weak solution 
of a stochastic differential equation with reflection of 
the form  
\begin{eqnarray}
&X(t)=X(0)+\int_0^tb(X(s))ds+\int_0^t\sigma (X(s))dW(s)+\int_0^t\gamma 
(s)\,d\lambda (s),\quad t\geq 0,&\non\\
&\quad\gamma (t)\in G(X(t)),\quad |\gamma (t)|=1,\quad d\lambda -
a.e.,\quad t\geq 0,&\label{SDER}\\
&X(t)\in\bar {E}_0,\quad\lambda (t)=\int_0^t{\bf 1}_{\partial E_0}
(X(s))d\lambda (s),\quad t\geq 0.&\non\end{eqnarray}

\begin{definition}
$X$, defined on some probability space,
is a weak solution of $(\ref{SDER})$ if there are 
$\lambda$ a.s.  continuous and nondecreasing, $\gamma$ a.s measurable 
and a standard Brownian motion $W$, all 
defined on the same probability space as $X$, 
such that $(X,\gamma ,\lambda )$ is compatible with $W$ (i.e. 
$W(t+\cdot )-W(t)$ is independent of ${\cal F}_t^{W,X,\gamma ,\lambda}$, where 
$\{{\cal F}^{W,X,\gamma ,\lambda}_t\}$ is the filtration generated by 
$(W,X,\gamma ,\lambda )$) and $(\ref{SDER})$ is satisfied. 
\end{definition}

\begin{theorem}\label{equiv}
Every weak solution of $(\ref{SDER})$ is a 
natural solution of the constrained martingale problem 
for $(A,E_0,B,\Xi )$ defined by $(\ref{rfclmp})$. 

Conversely, for every natural solution, $X$, of 
the constrained martingale problem 
for $(A,E_0,B,\Xi )$ there exists a weak solution of 
$(\ref{SDER})$ with the same distribution as $X$.
\end{theorem}

\begin{proof}
Let $X$ be a weak solution of $(\ref{SDER})$. Setting 
\[\Lambda ([0,t]\times C)\equiv\int_0^t{\bf 1}_C(X(s),\gamma (s))
d\lambda (s),\quad C\in {\cal B}(\Xi ),\]
we see that $X$ is a solution of the constrained 
martingale problem for $(A,E_0,B,\Xi )$. Since 
$\Lambda ([0,\cdot ]\times\Xi )$ is continuous and $(\ref{concent}
)$ is 
satisfied, by Proposition \ref{nat-char}, $X$ is a natural 
solution. 

Conversely, let $X=Y\circ\tau$, where $(Y,\lambda_0,\Lambda_1)$ is a solution of the 
controlled martingale problem $(A,E_0,B,\Xi )$ with filtration 
$\{{\cal F}_t\}$ and $\tau$ is given by $(\ref{tau})$. Without loss of 
generality we can suppose $\{{\cal F}_t\}$ complete. Then (see 
\cite{Kur91}, page 141) there is a 
$\{{\cal F}_t\}$-predictable, ${\cal P}(U)$-valued process $L$ such that, in 
particular, 
\[\int_{[0,t]\times U}u\,\Lambda_1(ds\times du)=\int_0^t\int_Uu\,
L(s,du)\,d\lambda_1(s)=\int_0^t\int_{U\cap G(Y(s))}u\,L(s,du)\,d\lambda_
1(s).\]
Note that $\big|\int_{U\cap G(Y(s))}u\,L(s,du)\big|>0$ $d\lambda_
1$-a.e. 
by Condition \ref{dircond} b) and $(2.4)$ of \cite{Kur91}. Then, setting 
\[\tilde{\gamma }(s)\equiv\frac {\int_{U\cap G(Y(s))}u\,L(s,du)}{\big
|\int_{U\cap G(Y(s))}u\,L(s,du)\big|},\quad\tilde{\lambda}_1(t)\equiv
\int_0^t\big|\int_{U\cap G(Y(s))}u\,L(s,du)\big|d\lambda_1(s),\]
we see that $(\ref{Y})$ can be written as 
\[Y(t)=Y(0)+\int_0^tb(Y(s))d\lambda_0(s)+\int_0^t\tilde{\gamma }(
s)d\tilde{\lambda}_1(s)+M(t).\]
By Lemma \ref{takeoff} and Lemma 
\ref{increasing}, $\lambda_0$ is strictly increasing, therefore 
$\tau =(\lambda_0)^{-1}$ and $X$ satisfies 
\[X(t)=X(0)+\int_0^tb(X(s))ds+\int_0^t\gamma (s)d\lambda (s)+N(t)
,\]
where $\gamma\equiv\tilde{\gamma}\circ\tau$, $\lambda\equiv\tilde{
\lambda}_1\circ\tau$ and $N\equiv M\circ\tau$ is a continuous martingale with 
$[N](t)=\int_0^t(\sigma\sigma^T)(X(s))ds$. Then the assertion follows by 
classical arguments. 
\end{proof}

\begin{theorem}
For each $\nu\in {\cal P}(\bar {E}_0)$, there exists a strong Markov solution 
of $(\ref{SDER})$.  If uniqueness in distribution holds 
among strong Markov solutions of $(\ref{SDER})$, then it 
holds among all solutions.  
\end{theorem}

\begin{proof}
The assertion follows from Theorem \ref{rfcmp-exist}, 
Theorem \ref{equiv}, Corollary \ref{M-const} and 
Corollary \ref{Muniq}. 
\end{proof}

We conclude this section with the proof of the 
equivalence between the controlled martingale problem 
$(\ref{rfclmp})$ and the corresponding patchwork martingale 
problem $(\ref{rfpwmp})$ (see Definition \ref{pwmp}). This equivalence is 
a valuable tool. For instance, in the last step of the proof of Theorem 
\ref{ctrl-exist}) we have already used one direction of the 
equivalence, which is immediate to see, 
namely the fact that  every solution of the 
patchwork martingale problem yields a solution of the 
controlled martingale problem. 
On the contrary, the other direction of the equivalence 
is nontrivial and is proved in the following theorem.

\begin{theorem}\label{mgpeqiv}
For every solution $(Y,\lambda_0,\Lambda_1)$ of the controlled martingale problem for 
$(A,E_0,B,\Xi )$ defined by $(\ref{rfclmp})$ there exist $l_1,...,l_
m$ 
such that $(Y,\lambda_0,l_1,...,l_m)$ is 
a solution of the patchwork martingale for $(A,E_0,B_1,E_1,...B_m
,E_m)$ 
defined by $(\ref{rfpwmp})$. 
\end{theorem}

\begin{proof}
First, we show that there is a Borel  
mapping $\Theta :\Xi\rightarrow \{\eta\in [0,\infty )^m:\,\sum_{i
=1}^m\eta_i=1\}$ such that 
\[u=\sum_{i=1}^m\Theta_i(x,u)g^i(x),\quad\Theta_i(x,u)=0\mbox{\rm \ for }
i\notin I(x).\]
Let $\mathbb G(x)\equiv [g_1(x),\ldots ,g_m(x))]$, and let $\mathbb 
G^{+}(x)$ be the 
Moore-Penrose pseudo-inverse (see \cite{BG03}, Chapter 
1). $\mathbb G^{+}(x)$ is a Borel function of $\mathbb G(x)$, hence of $
x$. 
Then, for each $u\in\R^d$ such that $\mathbb G(x)\eta =u$ has at least 
one solution , all solutions have the form
\[\eta (w)=\mathbb G^{+}(x)u+(I-\mathbb G^{+}(x)\mathbb G(x))w,\quad 
w\in\mathbb R^m.\]
For $(x,u)\in\Xi$, let $w^0(x,u):=\mbox{\rm argmin}|w|$, where the minimum is taken over all 
$w$ such that $\eta_i(w)\geq 0$, $i=1,...,m$, 
$\eta_i(w)=0$, $i\notin I(x)$, $\sum_{i=1}^m\eta_i(w)=1$. $w^0$ is a Borel 
function (\cite{DB19}). Then the mapping 
\[\Theta (x,u)\equiv\mathbb G^{+}(x)u+(I-\mathbb G^{+}(x)\mathbb 
G(x))w^0(x,u)\]
has the desired properties.

The assertion follows by defining 
\begin{equation}l_i(t)=\int_{[0,t]\times U}\Theta_i(Y(s),u)\Lambda_
1(ds\times du),\quad i=1,\ldots ,m.\label{ctrl-to-pwk}\end{equation}
\end{proof}

\setcounter{equation}{0}

\section{Examples of application to other boundary 
conditions}\label{applications}

\subsection{Non-local boundary conditions} Let 
$A\subset C(E)\times C(E)$ with ${\cal D}(A)$ dense in $C(E)$, and assume that 
there exist solutions of the martingale problem for $A$ 
with sample paths in $D_E[0,\infty )$ for all initial distributions 
$\nu\in {\cal P}(E)$.  

Let $U\equiv \{1\}$ and $B$ be defined by 
\[Bf(x,1)\equiv Bf(x)\equiv\int (f(y)-f(x))\eta (x,dy),\]
where $\eta$ is a transition function on $E$ and, for all $x\in E$, 
\[\eta (x,E_0)=\eta (x,E)=1.\]
Then the controlled martingale problem requires
\[f(Y(t))-f(Y(0))-\int_0^tAf(Y(s))d\lambda_0(s)-\int_0^tBf(Y(s))d
\lambda_1(s)\]
to be a martingale.  Note that the assumption that 
$\eta (x,E_0)=1$ implies that for every solution of the 
controlled martingale problem $P\{\tau (0)<\infty \}=1$. In fact, if 
$Y(0)\in E_0^c$, $P\{\tau (0)>t\}\leq e^{-t}$, since $B$ is a generator of a pure 
jump process with unit exponential holding times.  
Consequently, by Lemma \ref{laminf}, $\lambda_0(t)\rightarrow\infty$.

Processes of this type have been considered in a variety 
of settings, for example \cite{DN90,SS94}.  Semigroups 
corresponding to processes with nonlocal boundary 
conditions of this type have been considered in 
\cite{AKK16}.  Related models are considered in 
\cite{MR85}.

\subsection{Wentzell boundary conditions}\label{Wentz}
Let $A\subset C(E)\times C(E)$ and $B\subset C(E)\times C(E)$ be generators such 
that for every $\mu\in {\cal P}(E)$  there exist solutions of the 
martingale problem for $(A,\mu )$ and $(B,\mu )$, every solution  of 
the martingale problem for $A$ has continuous sample 
paths and every solution for $B$ has cadlag sample paths.  
In addition, assume that if $Z$ is a solution of the martingale 
problem for $B$ with $Z(0)\in\bar {E}_0$, then $Z(t)\in E_0$ for all 
$t>0$. Set $U\equiv \{1\}$ and $Bf(\cdot ,1)\equiv Bf$. 

Let $\mu\in {\cal P}(\bar {E}_0)$, let $Y^{\epsilon}(0)$ have distribution $
\mu$ and let $Y^{\epsilon}$ 
evolve as a solution of the martingale problem for $A$ 
until the first time $\tau_1^{\epsilon}$ that $Y^{\epsilon}$ hits $
\partial E_0$.  After time $\tau_1^{\epsilon}$, 
let $Y^{\epsilon}$ evolve as a solution of the martingale problem for 
$B$ until $\sigma_1^{\epsilon}\equiv\inf\{t>\tau_1^{\epsilon}:$ $\inf_{
x\in\partial E_0}|Y^{\epsilon}(t)-x|\geq\epsilon \}$.  
Recursively define $\tau_k^{\epsilon}$ and $\sigma_k^{\epsilon}$ and assume $
\sigma_0^{\epsilon}=0$.  By 
pasting, $Y^{\epsilon}$ is constructed so that for $f\in {\cal D}$, 
\[f(Y^{\epsilon}(t))-f(Y(0))-\int_0^t\left(\sum_{k=0}^{\infty}{\bf 1}_{
[\sigma_k^{\epsilon},\tau_{k+1}^{\epsilon})}(s)Af(Y^{\epsilon}(s)
)+\sum_{k=1}^{\infty}{\bf 1}_{[\tau_k^{\epsilon},\sigma_k^{\epsilon}
)}(s)Bf(Y^{\epsilon}(s))\right)ds\]
is a martingale.  Define
\[\lambda^{\epsilon}_0(t)=\int_0^t\sum_{k=0}^{\infty}{\bf 1}_{[\sigma_
k^{\epsilon},\tau_{k+1}^{\epsilon})}(s)ds.\]
Assume that ${\cal D}={\cal D}(A)={\cal D}(B)$ is dense in $C(E)$.  Then, by 
Theorem 3.9.4 of \cite{EK86},
$\{(Y^{\epsilon},\lambda_0^{\epsilon},\lambda_1^{\epsilon}),\epsilon 
>0\}$ is relatively compact, and every limit point 
$(Y,\lambda_0,\lambda_1)$ will give a solution of the controlled martingale problem, 
that is, for every $f\in {\cal D}$
\[f(Y(t))-f(Y(0))-\int_0^tAf(Y(s))d\lambda_0(s)-\int_0^tBf(Y(s))d
\lambda_1(s)\]
is a $\{{\cal F}_t^{(Y,\lambda_0,\lambda_1)}\}$-martingale. 

Our assumptions imply that $\lambda_0$ is strictly increasing, so 
$\lambda_0(t)\rightarrow\infty$ by Lemma \ref{laminf}.

Diffusions with Wentzell boundary conditions have been 
studied in
\cite{Wat71a,Wat71b,And76}.
Note that \cite{Wat71a,Wat71b} 
study the models using stochastic differential equations 
while \cite{And76} uses submartingale problems.  
\cite{Gra88} formulates what we call the constrained 
martingale problem.

\bibliography{martprob}

\begin{thebibliography}{37}
\providecommand{\natexlab}[1]{#1}
\providecommand{\url}[1]{\texttt{#1}}
\expandafter\ifx\csname urlstyle\endcsname\relax
  \providecommand{\doi}[1]{doi: #1}\else
  \providecommand{\doi}{doi: \begingroup \urlstyle{rm}\Url}\fi

\bibitem[Anderson(1976)]{And76}
Robert~F. Anderson.
\newblock Diffusions with second order boundary conditions. {I},{II}.
\newblock \emph{Indiana Univ. Math. J.}, 25\penalty0 (4):\penalty0
  367--395,403--441, 1976.
\newblock ISSN 0022-2518.

\bibitem[Arendt et~al.(2016)Arendt, Kunkel, and Kunze]{AKK16}
Wolfgang Arendt, Stefan Kunkel, and Markus Kunze.
\newblock Diffusion with nonlocal boundary conditions.
\newblock \emph{J. Funct. Anal.}, 270\penalty0 (7):\penalty0 2483--2507, 2016.
\newblock ISSN 0022-1236.
\newblock \doi{10.1016/j.jfa.2016.01.025}.
\newblock URL \url{https://doi-org/10.1016/j.jfa.2016.01.025}.

\bibitem[Ben-Israel and Greville(2003)]{BG03}
Adi Ben-Israel and Thomas N.~E. Greville.
\newblock \emph{Generalized inverses}, volume~15 of \emph{CMS Books in
  Mathematics/Ouvrages de Math\'{e}matiques de la SMC}.
\newblock Springer-Verlag, New York, 2003.
\newblock ISBN 0-387-00293-6.

\bibitem[Costantini(1992)]{Cos92}
C.~Costantini.
\newblock The {S}korohod oblique reflection problem in domains with corners and
  application to stochastic differential equations.
\newblock \emph{Probab. Theory Related Fields}, 91\penalty0 (1):\penalty0
  43--70, 1992.
\newblock ISSN 0178-8051.
\newblock \doi{10.1007/BF01194489}.
\newblock URL \url{http://dx.doi.org/10.1007/BF01194489}.

\bibitem[Costantini and Kurtz(2006)]{CK06}
Cristina Costantini and Thomas~G. Kurtz.
\newblock Diffusion approximation for transport processes with general
  reflection boundary conditions.
\newblock \emph{Math. Models Methods Appl. Sci.}, 16\penalty0 (5):\penalty0
  717--762, 2006.
\newblock ISSN 0218-2025.
\newblock \doi{10.1142/S0218202506001339}.
\newblock URL \url{http://dx.doi.org/10.1142/S0218202506001339}.

\bibitem[Costantini and Kurtz(2015)]{CK15}
Cristina Costantini and Thomas~G. Kurtz.
\newblock Viscosity methods giving uniqueness for martingale problems.
\newblock \emph{Electron. J. Probab.}, 20:\penalty0 no. 67, 27, 2015.
\newblock ISSN 1083-6489.
\newblock \doi{10.1214/EJP.v20-3624}.
\newblock URL \url{http://dx.doi.org/10.1214/EJP.v20-3624}.

\bibitem[Costantini and Kurtz(2018)]{CK18}
Cristina Costantini and Thomas~G. Kurtz.
\newblock Existence and uniqueness of reflecting diffusions in cusps.
\newblock \emph{Electron. J. Probab.}, 23:\penalty0 Paper No. 84, 21, 2018.
\newblock \doi{10.1214/18-EJP204}.
\newblock URL \url{https://doi.org/10.1214/18-EJP204}.

\bibitem[Costantini et~al.(2012)Costantini, Papi, and D'Ippoliti]{CPD12}
Cristina Costantini, Marco Papi, and Fernanda D'Ippoliti.
\newblock Singular risk-neutral valuation equations.
\newblock \emph{Finance Stoch.}, 16\penalty0 (2):\penalty0 249--274, 2012.
\newblock ISSN 0949-2984.
\newblock \doi{10.1007/s00780-011-0166-8}.
\newblock URL \url{http://dx.doi.org/10.1007/s00780-011-0166-8}.

\bibitem[Crandall et~al.(1992)Crandall, Ishii, and Lions]{CIL92}
Michael~G. Crandall, Hitoshi Ishii, and Pierre-Louis Lions.
\newblock User's guide to viscosity solutions of second order partial
  differential equations.
\newblock \emph{Bull. Amer. Math. Soc. (N.S.)}, 27\penalty0 (1):\penalty0
  1--67, 1992.
\newblock ISSN 0273-0979.
\newblock \doi{10.1090/S0273-0979-1992-00266-5}.
\newblock URL \url{http://dx.doi.org/10.1090/S0273-0979-1992-00266-5}.

\bibitem[Dai and Williams(1995)]{DW95}
J.~G. Dai and R.~J. Williams.
\newblock Existence and uniqueness of semimartingale reflecting {B}rownian
  motions in convex polyhedrons.
\newblock \emph{Theory Probab. Appl.}, 40\penalty0 (1):\penalty0 1--40, 1995.
\newblock ISSN 0040-361X.
\newblock \doi{10.1137/1140001}.
\newblock URL \url{https://doi.org/10.1137/1140001}.

\bibitem[Davis and Norman(1990)]{DN90}
M.~H.~A. Davis and A.~R. Norman.
\newblock Portfolio selection with transaction costs.
\newblock \emph{Math. Oper. Res.}, 15\penalty0 (4):\penalty0 676--713, 1990.
\newblock ISSN 0364-765X.
\newblock \doi{10.1287/moor.15.4.676}.
\newblock URL \url{https://doi-org/10.1287/moor.15.4.676}.

\bibitem[Di~Biase(2019)]{DB19}
Fausto Di~Biase.
\newblock Private communication, 2019.

\bibitem[Dupuis and Ishii(1993)]{DI93}
Paul Dupuis and Hitoshi Ishii.
\newblock S{DE}s with oblique reflection on nonsmooth domains.
\newblock \emph{Ann. Probab.}, 21\penalty0 (1):\penalty0 554--580, 1993.
\newblock ISSN 0091-1798.
\newblock URL
  \url{http://links.jstor.org/sici?sici=0091-1798(199301)21:1<554:SWORON>2.0.CO;2-D}.

\bibitem[Ethier and Kurtz(1986)]{EK86}
Stewart~N. Ethier and Thomas~G. Kurtz.
\newblock \emph{Markov {P}rocesses: {C}haracterization and {C}onvergence}.
\newblock Wiley Series in Probability and Mathematical Statistics: Probability
  and Mathematical Statistics. John Wiley \& Sons Inc., New York, 1986.
\newblock ISBN 0-471-08186-8.

\bibitem[Graham(1988)]{Gra88}
Carl Graham.
\newblock The martingale problem with sticky reflection conditions, and a
  system of particles interacting at the boundary.
\newblock \emph{Ann. Inst. H. Poincar\'{e} Probab. Statist.}, 24\penalty0
  (1):\penalty0 45--72, 1988.
\newblock ISSN 0246-0203.
\newblock URL \url{http://www.numdam.org/item?id=AIHPB_1988__24_1_45_0}.

\bibitem[Gray and Griffeath(1977)]{GG77}
Lawrence Gray and David Griffeath.
\newblock Unpublished manuscript, 1977.

\bibitem[Kang and Williams(2012)]{KW12}
W.~N. Kang and R.~J. Williams.
\newblock Diffusion approximation for an input-queued switch operating under a
  maximum weight matching policy.
\newblock \emph{Stoch. Syst.}, 2\penalty0 (2):\penalty0 277--321, 2012.
\newblock ISSN 1946-5238.
\newblock \doi{10.1214/12-SSY061}.
\newblock URL \url{https://doi.org/10.1214/12-SSY061}.

\bibitem[Kang et~al.(2009)Kang, Kelly, Lee, and Williams]{KKLW09}
W.~N. Kang, F.~P. Kelly, N.~H. Lee, and R.~J. Williams.
\newblock State space collapse and diffusion approximation for a network
  operating under a fair bandwidth sharing policy.
\newblock \emph{Ann. Appl. Probab.}, 19\penalty0 (5):\penalty0 1719--1780,
  2009.
\newblock ISSN 1050-5164.
\newblock \doi{10.1214/08-AAP591}.
\newblock URL \url{https://doi.org/10.1214/08-AAP591}.

\bibitem[Kang and Ramanan(2010)]{KR10}
Weining Kang and Kavita Ramanan.
\newblock A {D}irichlet process characterization of a class of reflected
  diffusions.
\newblock \emph{Ann. Probab.}, 38\penalty0 (3):\penalty0 1062--1105, 2010.
\newblock ISSN 0091-1798.
\newblock \doi{10.1214/09-AOP487}.
\newblock URL \url{https://doi.org/10.1214/09-AOP487}.

\bibitem[Kang and Ramanan(2014)]{KR14}
Weining Kang and Kavita Ramanan.
\newblock Characterization of stationary distributions of reflected diffusions.
\newblock \emph{Ann. Appl. Probab.}, 24\penalty0 (4):\penalty0 1329--1374,
  2014.
\newblock ISSN 1050-5164.
\newblock \doi{10.1214/13-AAP947}.
\newblock URL \url{https://doi-org/10.1214/13-AAP947}.

\bibitem[Kang and Ramanan(2017)]{KR17}
Weining Kang and Kavita Ramanan.
\newblock On the submartingale problem for reflected diffusions in domains with
  piecewise smooth boundaries.
\newblock \emph{Ann. Probab.}, 45\penalty0 (1):\penalty0 404--468, 2017.
\newblock ISSN 0091-1798.
\newblock \doi{10.1214/16-AOP1153}.
\newblock URL \url{https://doi-org/10.1214/16-AOP1153}.

\bibitem[Krylov(1973)]{Kry73}
N.~V. Krylov.
\newblock The selection of a {M}arkov process from a {M}arkov system of
  processes, and the construction of quasidiffusion processes.
\newblock \emph{Izv. Akad. Nauk SSSR Ser. Mat.}, 37:\penalty0 691--708, 1973.
\newblock ISSN 0373-2436.

\bibitem[Kurtz(1990)]{Kur90}
Thomas~G. Kurtz.
\newblock Martingale problems for constrained {M}arkov processes.
\newblock In \emph{Recent advances in stochastic calculus (College Park, MD,
  1987)}, Progr. Automat. Info. Systems, pages 151--168. Springer, New York,
  1990.

\bibitem[Kurtz(1991)]{Kur91}
Thomas~G. Kurtz.
\newblock A control formulation for constrained {M}arkov processes.
\newblock In \emph{Mathematics of random media (Blacksburg, VA, 1989)},
  volume~27 of \emph{Lectures in Appl. Math.}, pages 139--150. Amer. Math.
  Soc., Providence, RI, 1991.

\bibitem[Kurtz and Stockbridge(2001)]{KS01}
Thomas~G. Kurtz and Richard~H. Stockbridge.
\newblock Stationary solutions and forward equations for controlled and
  singular martingale problems.
\newblock \emph{Electron. J. Probab.}, 6:\penalty0 no. 17, 52 pp. (electronic),
  2001.
\newblock ISSN 1083-6489.

\bibitem[Kwon and Williams(1991)]{KW91}
Y.~Kwon and R.~J. Williams.
\newblock Reflected {B}rownian motion in a cone with radially homogeneous
  reflection field.
\newblock \emph{Trans. Amer. Math. Soc.}, 327\penalty0 (2):\penalty0 739--780,
  1991.
\newblock ISSN 0002-9947.
\newblock \doi{10.2307/2001821}.
\newblock URL \url{https://doi-org/10.2307/2001821}.

\bibitem[Lakner et~al.(2019)Lakner, Reed, and Zwart]{LRZ19}
Peter Lakner, Josh Reed, and Bert Zwart.
\newblock On the roughness of the paths of {RBM} in a wedge.
\newblock \emph{Ann. Inst. Henri Poincar\'{e} Probab. Stat.}, 55\penalty0
  (3):\penalty0 1566--1598, 2019.
\newblock ISSN 0246-0203.
\newblock \doi{10.1214/18-aihp928}.
\newblock URL \url{https://doi.org/10.1214/18-aihp928}.

\bibitem[Lions and Sznitman(1984)]{LS84}
P.-L. Lions and A.-S. Sznitman.
\newblock Stochastic differential equations with reflecting boundary
  conditions.
\newblock \emph{Comm. Pure Appl. Math.}, 37\penalty0 (4):\penalty0 511--537,
  1984.
\newblock ISSN 0010-3640.
\newblock \doi{10.1002/cpa.3160370408}.
\newblock URL \url{https://doi.org/10.1002/cpa.3160370408}.

\bibitem[Menaldi and Robin(1985)]{MR85}
Jos\'{e}-Luis Menaldi and Maurice Robin.
\newblock Reflected diffusion processes with jumps.
\newblock \emph{Ann. Probab.}, 13\penalty0 (2):\penalty0 319--341, 1985.
\newblock ISSN 0091-1798.
\newblock URL
  \url{http://links.jstor.org/sici?sici=0091-1798(198505)13:2<319:RDPWJ>2.0.CO;2-T&origin=MSN}.

\bibitem[Shreve and Soner(1994)]{SS94}
S.~E. Shreve and H.~M. Soner.
\newblock Optimal investment and consumption with transaction costs.
\newblock \emph{Ann. Appl. Probab.}, 4\penalty0 (3):\penalty0 609--692, 1994.
\newblock ISSN 1050-5164.
\newblock URL
  \url{http://links.jstor.org.ezproxy.library.wisc.edu/sici?sici=1050-5164(199408)4:3<609:OIACWT>2.0.CO;2-4&origin=MSN}.

\bibitem[Stroock and Varadhan(1971)]{SV71}
Daniel~W. Stroock and S.~R.~S. Varadhan.
\newblock Diffusion processes with boundary conditions.
\newblock \emph{Comm. Pure Appl. Math.}, 24:\penalty0 147--225, 1971.
\newblock ISSN 0010-3640.

\bibitem[Stroock and Varadhan(1979)]{SV79}
Daniel~W. Stroock and S.~R.~Srinivasa Varadhan.
\newblock \emph{Multidimensional diffusion processes}, volume 233 of
  \emph{Grundlehren der Mathematischen Wissenschaften [Fundamental Principles
  of Mathematical Sciences]}.
\newblock Springer-Verlag, Berlin, 1979.
\newblock ISBN 3-540-90353-4.

\bibitem[Tanaka(1979)]{Tan79}
Hiroshi Tanaka.
\newblock Stochastic differential equations with reflecting boundary condition
  in convex regions.
\newblock \emph{Hiroshima Math. J.}, 9\penalty0 (1):\penalty0 163--177, 1979.
\newblock ISSN 0018-2079.
\newblock URL \url{http://projecteuclid.org/euclid.hmj/1206135203}.

\bibitem[Taylor and Williams(1993)]{TW93}
L.~M. Taylor and R.~J. Williams.
\newblock Existence and uniqueness of semimartingale reflecting {B}rownian
  motions in an orthant.
\newblock \emph{Probab. Theory Related Fields}, 96\penalty0 (3):\penalty0
  283--317, 1993.
\newblock ISSN 0178-8051.
\newblock \doi{10.1007/BF01292674}.
\newblock URL \url{https://doi-org/10.1007/BF01292674}.

\bibitem[Watanabe(1971{\natexlab{a}})]{Wat71a}
Shinzo Watanabe.
\newblock On stochastic differential equations for multi-dimensional diffusion
  processes with boundary conditions.
\newblock \emph{J. Math. Kyoto Univ.}, 11:\penalty0 169--180,
  1971{\natexlab{a}}.
\newblock ISSN 0023-608X.
\newblock \doi{10.1215/kjm/1250523692}.
\newblock URL \url{https://doi-org/10.1215/kjm/1250523692}.

\bibitem[Watanabe(1971{\natexlab{b}})]{Wat71b}
Shinzo Watanabe.
\newblock On stochastic differential equations for multi-dimensional diffusion
  processes with boundary conditions. {II}.
\newblock \emph{J. Math. Kyoto Univ.}, 11:\penalty0 545--551,
  1971{\natexlab{b}}.
\newblock ISSN 0023-608X.
\newblock \doi{10.1215/kjm/1250523619}.
\newblock URL \url{https://doi-org/10.1215/kjm/1250523619}.

\bibitem[Weiss(1981)]{Wei81}
Alan~Arthur Weiss.
\newblock \emph{Invariant measures of diffusion processes on domains with
  boundaries}.
\newblock ProQuest LLC, Ann Arbor, MI, 1981.
\newblock URL
  \url{http://gateway.proquest.com/openurl?url_ver=Z39.88-2004&rft_val_fmt=info:ofi/fmt:kev:mtx:dissertation&res_dat=xri:pqdiss&rft_dat=xri:pqdiss:8127975}.
\newblock Thesis (Ph.D.)--New York University.

\end{thebibliography}

\vskip .3 in
\date{\noindent{\bf Acknowledgments.} We are grateful to Fausto Di Biase who provided a nice answer to a measurability issue which arose in the proof of Theorem \ref{mgpeqiv}.

This paper was completed while the second author was visiting the University of California, San Diego with the support of the Charles Lee Powell Foundation.  The hospitality of that institution, particularly that of Professor Ruth Williams, was greatly appreciated.  }

\end{document}